\documentclass[11pt,a4paper,leqno]{article}

\usepackage[a4paper, margin=2.5cm]{geometry}

\usepackage{amssymb,amsthm,amsmath,amsfonts,mathrsfs,bbm}
\usepackage[dvipsnames]{xcolor}

\numberwithin{equation}{section}

\usepackage{amsmath,amsfonts,amssymb}
\usepackage{mathtools}
\usepackage[english]{babel}
\usepackage{polski}
\usepackage{textcomp}
\usepackage{fancyhdr}
\usepackage{eso-pic}
\usepackage{color}
\usepackage{xcolor}
\usepackage{stmaryrd}
\usepackage{comment}
\usepackage{appendix}
\usepackage{soul}
\usepackage{cancel}

\usepackage[colorlinks=true]{hyperref}
\usepackage{cleveref}

\DeclareMathOperator\dvg{div}

\newtheorem{thm}{Theorem}

\newtheorem{prop}[thm]{Proposition}

\newtheorem{rmq}[thm]{Remark}

\newtheorem{lem}[thm]{Lemma}
\newtheorem{example}[thm]{Example}

\numberwithin{thm}{section}

\newcommand{\ro}{\rho}
\newcommand{\n}{\nabla}

\newcommand{\intT}{\int_{\mathbb{T}^2}}
\newcommand{\T}{\mathbb{T}}
\newcommand{\R}{\mathbb{R}}
\newcommand{\esp}[1]{ \quad \text{#1} \quad}

\newcommand{\normes}[2]{\left\Vert #1\right\Vert_{L^{#2}}}
\newcommand{\normede}[1]{\left\Vert #1\right\Vert_{2}}
\newcommand{\normep}[1]{\left\Vert #1\right\Vert_{p}}
\newcommand{\normer}[1]{\left\Vert #1\right\Vert_{r}}
\newcommand{\normeinf}[1]{\left\Vert #1\right\Vert_{L^\infty}}

\newcommand{\dd}{\mathrm{d}}

\title{Inhomogenous Navier--Stokes equations with unbounded density}
\author{Jean-Paul Adogbo$^{\S}$, Piotr B. Mucha$^*$, Maja Szlenk$^{*,\dagger}$}
\date{\today}

\begin{document}

\maketitle

{
\footnotesize

\centerline{$^\S\;$Ceremade, UMR CNRS 7534, Université Paris Dauphine-PSL }
\centerline{Place du Mar\' echal De Lattre De Tassigny 75775 Paris cedex 16 (France)}
\bigbreak
\centerline{$^*\;$Institute of Applied Mathematics and Mechanics, University of Warsaw, }
\centerline{ul. Banacha 2, Warsaw 02-097, Poland}

\bigbreak
\centerline{$^\dagger\;$LAMA, UMR CNRS 5127, Universit\'e Savoie Mont Blanc, }
\centerline{B\^{a}t. Le Chablais, Campus Scientifique 73376 Le Bourget du Lac, France}
}

\bigbreak

\begin{abstract}

In the current state of the art regarding the Navier--Stokes equations, the existence of unique solutions for incompressible flows in two spatial dimensions is already well-established.
Recently, these results have been extended to models with variable density, maintaining positive outcomes for merely bounded densities, even in cases with large vacuum regions. However, the study of incompressible Navier-Stokes equations with unbounded densities remains incomplete. Addressing this gap is the focus of the present paper.

Our main result demonstrates the global existence of a unique solution for flows initiated by unbounded density, whose regularity/integrability is characterized within a specific subset of the Yudovich class of unbounded functions. The core of our proof lies in the application of Desjardins' inequality, combined with a blow-up criterion for ordinary differential equations. Furthermore, we derive time-weighted estimates that guarantee the existence of a $C^1$ velocity field and ensure the equivalence of Eulerian and Lagrangian formulations of the equations. Finally, by leveraging results from \cite{DanMu}, we conclude the uniqueness of the solution.

\end{abstract}


\section{Introduction}
A vast amount of research has focused on the mathematical analysis of the \textit{incompressible inhomogeneous Navier--Stokes} equations with bounded initial density. This system models a fluid formed by mixing two incompressible, miscible fluids with different densities. For the derivation of the model, we refer to \cite{lions1996vol1}. Common example involves for example river dynamics, in particular junction of water channels \cite{ammar20062d}. Another inhomogeneous fluid is blood, which is a suspension of blood cells in plasma \cite{Fusi_blood,archer2009dynamical,liu2004coupling}. 
Recall that these equations read
\begin{equation}\label{INS}
   \tag{\textbf{INS}}
   \begin{cases}
       \ro_t+v\cdot\n \ro=0 \quad \text{in } \mathbb{R}_+\times \Omega ,\\
       \ro v_t+\ro v\cdot\n  v-\mu \Delta v+\n P=0 \text{  in } \mathbb{R}_+\times \Omega, \\
       \dvg v =0 \text{  in  } \mathbb{R}_+\times \Omega.
   \end{cases}
\end{equation}
The unknowns are the density $\ro=\ro(t,x)$, the velocity $v=v(t,x)$ and the pressure $P=P(t,x)$. 

The recent analysis of the above system motivates us to give the answer to the following  natural question:

\begin{center}


{\it How much can we extend the class of admissible densities, \\[7pt]
 to derive the same result as assuming that the density is bounded?}

\end{center}

The most known and natural choice for extending the $L^\infty$ framework is the $BMO$ space. Here we mention the branch of results related to the so-called logarithmic Sobolev inequality \cite{Dan1996,KozTan} and marginal spaces \cite{Rosa,Mu-tra,MuRu}. In this paper, our aim is to extend the result from \cite{DanMu} to a larger class of densities, in particular containing unbounded functions. Following the strategy from the mentioned result, the $BMO$ space seems to be not a suitable class for the initial density. However, we are able to obtain the result in a narrower regularity class, but still larger than $L^\infty$. 
We assume that the fluid domain $\Omega$ is either the torus  $\T^2$ or a $C^2$ simply connected bounded domain of $\R^2$. The system $(\ref{INS})$ is supplemented with initial data at time $t=0$:
\begin{equation}
    \label{ini:data}
    \ro|_{t=0}=\ro_0 \esp{and} v|_{t=0}=v_0.
\end{equation}
In the case where $\Omega$ is not a torus, we additionally supplement the system with the boundary condition $v=0$ on $\partial\Omega$.
It is well-known that sufficiently smooth solutions to (\ref{INS}) possess several important properties, such as
\begin{description}
    
 \item[$\bullet$ Conservation of the momentum :] (for the torus case only)
\begin{equation}
    \label{conser:momen}
     \intT (\ro v)(t,x)dx=\intT (\ro_0 v_0)(x)dx,
\end{equation}
  \item[$\bullet$ Conservation of total mass:] 
\begin{equation}
    \label{conser:mass}
  \int_\Omega \ro(t,x)dx=\int_\Omega \ro_0(x)dx,
\end{equation}
\item[$\bullet$ Energy balance:]   
\begin{equation}
    \label{energy:balance}
    \frac{1}{2}\frac{d}{dt}\int_{\Omega} \ro |v|^2+\mu\int_{\Omega}|\n v|^2 =0,
\end{equation}
   \item[$\bullet$ Conservation of Lebesgue norm:] for all $p\in [1,\infty)$, 
   \begin{equation}
    \label{conser:BMO:Lp}
    \int_\Omega\ro^p(t,x)\;dx=\int_\Omega \ro_0^p(x)\;dx.
\end{equation}
\end{description}
The equations (\ref{conser:momen}) and (\ref{conser:mass}) follow straightforward  from integrating the equation over $\Omega$, whereas (\ref{energy:balance}) is derived by testing the momentum equation by $v$.
Furthermore, calculating the equation for $\ro^p$ we get
\[ \partial_t\ro^p + \dvg(\ro^pv) = 0, \]
 which leads to identity \eqref{conser:BMO:Lp}.

When the density $\ro$ of the solution is constant (corresponding to a monofluid), (\textbf{INS}) simplifies to the so-called homogeneous Navier-Stokes system, which is written as follows:
\begin{equation}\label{NS}
   \tag{\textbf{NS}}
   \begin{cases}
        v_t+ v\cdot\n  v-\mu \Delta v+\n P=0 \text{  in } \mathbb{R}_+\times \Omega, \\
       \dvg v =0 \text{  in  } \mathbb{R}_+\times \Omega.
   \end{cases}
\end{equation}
The system (\ref{NS}) shares with (\ref{INS}) the properties of the energy balance given by \eqref{energy:balance} and the conservation of momentum described by \eqref{conser:momen}.



The global existence theory of (\ref{NS}) has been known since the work of Leray in 1934 \cite{leray1934mouvement}. In the case $\Omega = \mathbb{R}^3$, Leray proved that any divergence-free velocity $v_0$ belonging to $L^2$ generates a global distributional solution of (\ref{NS}). Moreover, it was shown that the constructed solution $v$ satisfies $v \in L^\infty([0,\infty); L^2) \cap L^2([0,\infty); \dot{H}^1)$.

Numerous papers have since been devoted to this topic in both bounded and unbounded domains for dimensions $d = 2$ and $d = 3$. In 1959, O.A. Ladyzhenskaya \cite{Ladyzhenskaya1959}, along with J.-L. Lions and G. Prodi \cite{lions1959theoreme}, established the existence and uniqueness of solutions in two dimensions. While the two-dimensional case is now much better understood, the question of uniqueness of finite energy solutions in the three-dimensional case ($d = 3$) remains unresolved to this day, known as the VI-th Millennium Problem.

Concerning the inhomogeneous Navier-Stokes equations, the current state of the art indicates that the weak (distributional) solution theory is quite similar to that of the homogeneous case. In fact, when the density is bounded away from zero, the energy balance \eqref{energy:balance} suggests seeking weak solutions within a framework similar to that of the homogeneous Navier-Stokes equations, and the strong solution theory aligns with the homogeneous case as well. In this direction, A. V. Kazhikhov \cite{kazhikhov1974solvability} obtained, for the first time in 1974, the existence of global-in-time weak solutions, with the energy balance replaced by an inequality. Later, in 1990, Simon \cite{simon1990nonhomogeneous} improved this result by removing the lower bound on the density, assuming only that $\rho$ is nonnegative. Subsequently, P.-L. Lions \cite{lions1996vol1} demonstrated that the density is a renormalized solution of the mass equation. This insight allowed him to also consider cases where the viscosity $\mu$ depends on $\rho$ (see also \cite{desjardins1997global}).

The aforementioned existence results did not address uniqueness, which was first established by Ladyzhenskaya and Solonnikov in 1978 \cite{ladyzhenskaya1978unique}. They considered the case where $\Omega$ is a smooth bounded domain of $\mathbb{R}^2$ or $\mathbb{R}^3$, with smooth enough data, and the density is bounded from above and away from zero. They proved the existence of a unique local-in-time smooth solution, which is global in the two-dimensional case, or in higher dimensions if the initial velocity is sufficiently small.

The uniqueness of solutions $(\rho, v)$ for the inhomogeneous Navier-Stokes system (\ref{INS}), starting from initial data in the energy space (i.e., with finite initial energy), remains a significant challenge. As Danchin notes in \cite{danchin2024global}, unlike the homogeneous Navier-Stokes system (\ref{NS}), where uniqueness is typically achieved through energy methods combined with Gronwall’s inequality, it is not clear how to prove the uniqueness of solutions to (\ref{INS}) without the crucial estimate
\begin{align}
    \label{n2:v:inLin:L2}
    \nabla v \in L^1_{loc}(0,T;L^\infty).
\end{align}
Thus, for $(\rho_1, v_1)$ and $(\rho_2, v_2)$ being two solutions to (\ref{INS}), the difference in the densities $\delta \rho = \rho_1 - \rho_2$ satisfies a transport equation governed by a divergence-free vector field. Additionally, assuming the densities are merely bounded, the source term of this transport equation has negative regularity with respect to the spatial variable (see also \cite{danchinMucha2023compressible, prange2024free, hoff2006uniqueness}). Consequently, the estimate \eqref{n2:v:inLin:L2} is crucial to ensure the existence of a unique flow and to control the propagation of negative regularity in the transport equation.

An alternative approach is to reformulate the system (\ref{INS}) in Lagrangian coordinates, as in \cite{DanMu}. In this case, the density remains constant along characteristic curves, and the velocity still satisfies a parabolic equation. However, establishing the equivalence between the Eulerian and Lagrangian formulations of (\ref{INS}) in this low-regularity setting still depends on the validity of \eqref{n2:v:inLin:L2}, a condition that cannot be guaranteed when $v_0$ is only in $L^2$. In fact, this inequality even fails for the heat flow.

  For the heat equation, inequality  \eqref{n2:v:inLin:L2}  holds if the initial velocity $v_0$, belongs to the appropriate functional space, such as:
\begin{description}
    \item[$\blacksquare\; L^2$ framework:] the Sobolev spaces $H^s$ with $ s>\frac{d}{2}-1$ or the homogeneous Besov space  $ \Dot{B}^{\frac{d}{2}-1}_{2,1}$. In the latter case,  $\n e^{t\Delta} v_0 \in L^1_{loc}(0,T;\Dot{B}^{\frac{d}{2}}_{2,1})$ which is
embedded in the set $L^1_{loc}(0,T; C_0(\mathbb{R}^d) )$ of continuous functions on $\mathbb{R}^d$ vanishing at infinity  (\cite{HajDanChe11}).
\item[$ \blacksquare\; L^p$ framework:] the homogeneous Besov spaces $ \Dot{B}^{\frac{d}{p}-1}_{p,1}$, with $1\le p\le \infty$ (\cite{HajDanChe11}).
\end{description}

This observation has inspired numerous works on the well-posedness theory for the inhomogeneous Navier-Stokes system (\ref{INS}). In \cite{ZhangZhang13}, the authors proved the existence and uniqueness of solutions of  (\ref{INS}) when $v_0\in H^s(\mathbb{R}^2)$, with $s>0$, and $\ro_0$ is bounded from above and from below ( i.e. $0<c_0\le \ro_0\le c_1<\infty$).



In the three-dimensional case, the authors constructed a unique local-in-time solution for initial data, which becomes global if the initial data satisfies a smallness condition. This result was later improved in \cite{ZhangChen16}, \cite{zhang2020global}. Recently, R. Danchin and S. Wang in \cite{DanWang2023global} established the global-in-time well-posedness for initial velocity $v_0$
  in the critical Besov space  $\Dot{B}_{p,1}^{\frac{d}{p}-1}$ with $p\in (1,d)$, and for $\ro_0$  close to a positive constant.

In the presence of vacuum, local and global well-posedness in bounded domains was demonstrated in \cite{DanMu}. As a by-product, they provided a complete answer to the question posed by Lions in his book \cite{lions1996vol1}, 
concerning the evolution of a "patches of density" (i.e drop of incompressible viscous fluid in vacuum). More recently, Prange and Tan studied the existence and uniqueness of solutions on $\mathbb{R}^d$ (with
$d=2,3$) under certain special vacuum configurations in \cite{prange2024free}.

The aim of this paper is to go beyond the result in \cite{DanMu} by considering unbounded, nonnegative initial density, while maintaining the regularity of the initial divergence free velocity $v_0$.
We observe that the method used in \cite{DanMu} can also be applied when the density is just in any Lebesgue space $L^p(\Omega)$ (with $1\le p<\infty$) yielding local existence and uniqueness.  A remarkable feature of our result is that, even though
the density is rough and does not need to be bounded, one can exhibit a framework in  which we get global existence and uniqueness result. Moreover, we succeed in exhibiting gain of regularity for the velocity, which entails that \eqref{n2:v:inLin:L2} holds, thereby leading to uniqueness.

\section{The Results}

In this paper, we construct a global in time, unique solution to the system (\ref{INS}) with the initial velocity belonging to $H^1_0$. Our assumptions on initial density involve a new functional space, which contains $L^\infty$ but also admits unbounded functions, provided that the blow-up is sufficiently controlled.

We assume that the initial data $(\ro_0,v_0)$ satisfy
\begin{equation}\label{initia-data}
v_0\in H_0^1(\Omega), \quad \dvg v_0=0 \quad \text{and} \quad \sqrt{\ro_0}v_0\in L^2(\Omega). 
\end{equation}
Moreover,  $\ro_0$ is such that 
\begin{equation}\label{initial-rho}
\ro_0\geq 0, \quad \int_\Omega\ro_0\; \dd x = M>0.
\end{equation}
To determine the space of functions describing the class of the initial data for the density we introduce the class $\mathcal{L}$, defined as follows
\begin{equation}\label{def-Lclass}
\mathcal{L} := \left\{f\in \bigcap_{p\geq 1}L^p(\Omega): \lim_{p\to\infty}\frac{1}{\log p}\|f\|_p^2\log(1+\|f\|_p)=0\right\}. 
\end{equation}
Obviously $L^\infty\subset\mathcal{L}$. On the other hand, one can show that $\mathcal{L}\subset BMO$. In the Appendix \ref{app_L} we analyze the key examples to explain the main features of this space.

\smallskip 

Let us now state our main result.
\begin{thm}
    \label{thm:Lp:density}
    Let $\Omega$ be a $C^2$ bounded subset of $\mathbb{R}^2$ or the torus $ \mathbb{T}^2 $. Assume that the initial data $(\ro_0,v_0)$ satisfy (see (\ref{def-Lclass}))
    \begin{align}
        \label{data:ass}
        \begin{split}
            & v_0\in H_0^1(\Omega), \quad \dvg v_0=0,\\
            & \ro_0\in \mathcal{L}, \quad \ro_0\geq 0, \quad M:= \int_\Omega \varrho_0\;\dd x >0.
        \end{split}
    \end{align}
\noindent   
    Then there exists a unique solution  $(\ro, v, \n P)$ to system (\ref{INS}) with data $(\ro_0, v_0)$ fulfilling the conservation of momentum (\ref{conser:momen}) (in the case $\Omega=\mathbb{T}^2$), the conservation of mass \eqref{conser:mass}, the balance energy  \eqref{energy:balance} and the conservation of Lebesgue norm \eqref{conser:BMO:Lp}.
    Moreover, the following properties of regularity hold true:
    \begin{align*}
      \ro \in L^\infty(\mathbb{R}_+,L^{p}(\Omega)) \esp{for all } 1\le p< \infty , \esp{}  v\in L^\infty(\mathbb{R}_+, H^1_0(\Omega)), \\
       \sqrt{\ro}v_t \in L^2(\mathbb{R}_+, L^2(\Omega)),   \esp{} \n^2 v, \n P\in L^2(\mathbb{R}_+,L^r(\Omega)), \esp{for all} r<2.
    \end{align*}
    In addition, for all $T>0$
    \begin{align*}
      &  \sqrt{t\ro}v_t \in L^\infty(0,T;L^2(\Omega)),\quad \sqrt{t}\n v_t\in L^2(0,T;L^2(\Omega)),\\
       & \n^2( \sqrt{t} v),\; \n(\sqrt{t}P) \in L^q(0,T;L^{\lambda}(\Omega)) \esp{for} 2\le q\le \infty \esp{and} \lambda< \frac{2q}{q-2} .
    \end{align*}
    Finally, we have $\sqrt{\ro}v\in C(\mathbb{R}_+; L^2(\Omega))$, $\ro\in C(\mathbb{R}_+;L^{p_1}(\Omega))$  and $ v\in H^\gamma(0,T; L^{p_2}(\Omega) )$, for all ${p_1}<\infty$, $p_2<\infty$, $\gamma<\frac{1}{2} $ and $T>0$.
\end{thm}

\begin{rmq}
      \label{rmq:no:smallnest}
     
      The exact condition for the initial data that we need states that
    \begin{equation}\label{new_initial} \lim_{p\to\infty} \frac{1}{\log p}\|\varrho_0\|_p^{2+\frac{6}{p-3}}\log^{1+\frac{2}{p-3}}\left( e+\frac{\normede{\ro-M}^2}{M^2}+\normep{\ro_0}\normede{\sqrt{\ro_0}v_0}^2\right) =0. \end{equation}
    However, since $\sqrt{\ro_0}v_0\in L^2(\Omega)$, this estimate is satisfied if $\ro_0\in\mathcal{L}$.
    
\end{rmq}

\begin{rmq}
Assuming only that $\ro_0\in L^p(\Omega)$ for some fixed $3<p<\infty$, we obtain a local existence result. In that case, the solution exists on $[0,\infty)$ provided that the initial data satisfy
\begin{multline*}
    \normep{\ro_0}^{\frac{2p}{p-3}}\normede{\sqrt{\ro_0}v_0}^{\frac{2(p-1)}{p-3}} \log^{\frac{p-1}{p-3}}\left( e+\frac{\normede{\ro-M}^2}{M^2}+\normep{\ro_0}\normede{\sqrt{\ro_0}v_0}^2\right)\|\sqrt{\rho_0}v_0\|_2^2\\
    < c_0'\int_{e}^\infty \frac{dx}{(x\log x)^{\frac{p-1}{p-3}}} 
\end{multline*}
for some positive constant $c_0'$ independent of the data. Alternatively, for arbitrary large initial data we obtain a solution on the interval $[0,T^*]$, where $T^*$ is such that
\begin{multline*} \normep{\ro_0}^{\frac{2p}{p-3}}\normede{\sqrt{\ro_0}v_0}^{\frac{2(p-1)}{p-3}} \log^{\frac{p-1}{p-3}}\left( e+\frac{\normede{\ro-M}^2}{M^2}+\normep{\ro_0}\normede{\sqrt{\ro_0}v_0}^2\right)\int_0^{T^*} \|\nabla v(\tau)\|_2^2 d\tau \\
< c_0'\int_e^\infty \frac{dx}{(x\log x)^{\frac{p-1}{p-3}}}. \end{multline*}

\end{rmq}

Our second result addresses the issue of global well-posedness of (\ref{INS}) when the initial density is bounded away from vacuum. Specifically, we assume that the initial density belongs to the following space:
\begin{align}
    \label{def:Yudo:null}
    \mathcal{Y}_0:= \left\{f\in \bigcap_{p\geq 1}L^p(\Omega): \lim_{p\to\infty}\frac{1}{ p}\|f\|_p=0\right\}. 
\end{align}
Note that both $\mathcal{L}$ and $\mathcal{Y}_0$ are the subspaces of some Yudovich spaces $\mathcal{Y}^{\Theta(p)}$ with suitable $\Theta(p)$. Defining the Yudovich space as (see e. g. \cite{crippaGiorgio24})
\[ \mathcal{Y}^{\Theta(p)}:=\left\{f\in\bigcap_{p\geq 1}L^p(\Omega): \sup_{p\geq 1}\frac{\|f\|_p}{\Theta(p)}<\infty \right\}, \]
we immediately see that $\mathcal{Y}_0\subset\mathcal{Y}^p$, and moreover $ \mathcal{Y}^{(\sqrt{\log p})^{1-\varepsilon}} \subset \mathcal{L}\subset\mathcal{Y}^{\sqrt{\log p}}$, for any $0<\varepsilon<1$.

In order to find a connection with some more classical functional spaces, we prove in Appendix \ref{app_L} that 
\begin{align}
    \label{emb:Y:L:BMO}
    L^\infty(\Omega)\subsetneq \mathcal{L}\subset  \mathcal{Y}_0 \subsetneq  L^{\exp}(\Omega).
\end{align}
Our second global existence and uniqueness statement reads as follows:
\begin{thm}
    \label{thm:Lp:ro>0}
    Let $\Omega$ be a $C^2$ bounded subset of $\mathbb{R}^2$ or the torus $ \mathbb{T}^2 $. Assume that the initial data $(\ro_0,v_0)$ satisfy 
    \begin{align}
        \label{data:ass:ro>0}
        \begin{split}
            & v_0\in H_0^1(\Omega), \quad \dvg v_0=0,\\
            & \ro_0\in \mathcal{Y}_0, \quad \ro_0\geq \ro_*>0, \quad M:= \int_\Omega \varrho_0\;\dd x >0.
        \end{split}
    \end{align}
  
    Then there exists a unique solution  $(\ro, v, \n P)$ to (\ref{INS}) with data $(\ro_0, v_0)$ fulfilling the conservation of momentum (\ref{conser:momen}) (in the case $\Omega=\mathbb{T}^2$), the conservation of mass \eqref{conser:mass}, the balance energy  \eqref{energy:balance} and the conservation of Lebesgue norm \eqref{conser:BMO:Lp}.
    
    Furthermore, the density $\rho$ satisfies the condition $\rho \geq \rho_*$, and the solution $(\rho, v, \nabla P)$ adheres to all regularity properties detailed in Theorem \ref{thm:Lp:density}.
   
\end{thm}

Let us report on the main ideas leading to Theorems \ref{thm:Lp:density} and \ref{thm:Lp:ro>0}. We focus here on the case $\Omega=\mathbb{T}^2$. Assuming
that we are given a solution $(\ro, v)$ to (\ref{INS}), the first step is to establish global-in-time a priori estimates for the $H^1$ norm of $v$ in terms of the data and of the parameters of the system.  The overall
strategy has some similarities with the work \cite{DanMu} dedicated to system (\ref{INS}) with bounded initial density and $v_0\in H^1$. Performing a basic energy method on the momentum equation, we will succeed in extracting some parabolic smoothing effect even if the density is rough and vanishes. This enables us to get a control on $\n v $ in $ L^\infty(L^2)$, $\sqrt{\ro}v_t $ in $L^2(L^2)$ and $ \n^2 v, \n P $ in $L^2(L^r)$  for $r<2$; it is worth mentioning that the restriction on $r$ comes from the fact that the density may be unbounded - for $\ro_0\in L^\infty(\Omega)$ one obtains $ \n^2 v, \n P $ in $L^2(L^2)$, as it was done in \cite{DanMu}. In fact, the proof of Theorem \ref{thm:Lp:ro>0} uses the same argument expected for propagating the Sobolev regularity.

The main tool to obtain the above regularity, if the density $\ro$ contains regions of vacuum (i.e $\ro\ge 0$),  is the following logarithmic interpolation inequality
\begin{multline}
    \label{Derjadin:ine}
    \left( \intT\ro v^4 \right)^{\frac{1}{2}} \le C_p \normede{\sqrt{\ro}v}\left|\intT\ro v\right|\\
        +C_p \normede{\sqrt{\ro}z} \normede{\n v} \log^{\frac{1}{2}}\left( e+\frac{\normede{\ro-M}^2}{M^2}+\normep{\ro}\frac{\normede{\n v}^2}{\normede{\sqrt{\ro}v}^2}\right),
\end{multline}
where $M$ stands for the average of the density $\ro$ and $p>1$. It is important to emphasize  that inequality \eqref{Derjadin:ine} has been discovered by B. Derjardins \cite{Der97} and is an appropriate substitute of the classical Ladyzhenskaya inequality for constant density
\begin{align*}
    \normes{v}{4}^2
\le C \normede{v}\normede{\n v}.
\end{align*}
To be more clear, after testing (\ref{INS})$_2$ by $v_t$,  we are able to get a priori global in time bound on $\n v$ in $L^\infty(L^2)$. The key is to estimate the trouble-making term $\normede{\sqrt{\ro} v\cdot \n v}^2$, which is handled differently, depending whether the density is distant from vacuum or not. Below we signal the differences between these two cases:
\begin{description}
    \item[$\bullet$ The density contains regions of vacuum ($\ro\ge 0$).] In this case, the estimate \eqref{Derjadin:ine} comes into play. In fact, thanks to H\"{o}lder and Young's inequality we get that
\begin{align*}
     \intT\ro |v\cdot\n v|^2 &\le   C_{r,p} \normede{\n v}^2  \normep{\ro}^{\frac{1}{2(1-\alpha)}} \normede{\sqrt{\ro}|v|^2}^\frac{1}{1-\alpha}+ \frac{1}{4C_r\normep{\ro}} \normes{\n^2 v}{r}^{2}
\end{align*}
with 
\begin{align*}
      \frac{1}{1-\alpha}=\dfrac{2-\frac{2}{p}}{1-\frac{3}{p}}.
\end{align*}
 The last term may be absorbed, and the first one may be handled by the inequality \eqref{Derjadin:ine}. Consequently, one gets that 
\begin{align*}
     \frac{d}{dt} \normede{\n v}^2+ c_0 \left(\normede{\sqrt{\ro}v}^2+\normer{\n P,\n^2 v}^2\right) \le \frac{1}{p}f\normep{\ro}\normede{\n v}^{\frac{1}{1-\alpha}} \log^{\frac{1}{2(1-\alpha)}}(e+\normede{\n v}^2),
\end{align*}
where $f\in L^1(\mathbb{R}_+)$ and $c_0>0$. 
Since $\frac{1}{1-\alpha}>2$, there appears a higher power of $\normede{\n v}^2$ in the right- hand side of the above inequality. In consequence, for fixed $p$ this estimate provides only a local estimate, which will go to infinity in finite time. However, we can control the blow-up with respect to $p$, and taking $p\to\infty$ we obtain a global estimate provided that (\ref{new_initial}) holds.
 \item[$\bullet$ The density is distant from vaccum ($\ro\ge \ro_*>0$).]  In this case, we automatically obtain the information that  $v\in L^\infty(L^2(\Omega))$, which follows directly from \eqref{energy:balance}. Hence combining Sobolev inequality and H\"{o}lder inequality, we have, for all $\varepsilon>0$

     \begin{align*}
        \intT \ro |v\cdot\n v|^2\le  \frac{\varepsilon}{\normep{\ro}}\normes{\n^2 v}{r}^2+ \varepsilon^{-\frac{2}{p-3}} 
        \normep{\ro}^{\frac{p+1}{p-3}}   \normes{v}{2}^{2\frac{p-1}{p(p-3)}}\normede{\n v}^{2(2+\varepsilon_p)} ,
    \end{align*}
  where   
$\varepsilon_p \simeq \frac{1}{p}$ as $p \to \infty$.  \\
Choosing suitably $\varepsilon$ we can absorb the first term in the right- hand side. At the end we arrive at, 
    \begin{align*}
      &  \frac{d}{dt}\normede{\n v}^2+ c_0 \left(\normede{\sqrt{\ro}v}^2+\normer{\n P,\n^2 v}^2\right) \le g(t)X^{1+\varepsilon_p}\\
        &\text{with   } g(t):=C  \normep{\ro}^{\frac{p+1}{p-3}}   \normes{v}{2}^{2\frac{p-1}{p(p-3)}}\normede{\n v}^2\in L^1(\mathbb{R}_+).
    \end{align*}
Solving this inequality and using the fact that the density $\ro$ satisfy $\normep{\ro}=o(p)$, as $p \to \infty$, we obtain a global estimate for the velocity  $v$ in $H^1$.
\end{description}

The next step is to recover the regularity \eqref{n2:v:inLin:L2}
which is essential to reformulate the system (\ref{INS}) in Lagrangian coordinates and in consequence get uniqueness of solutions.
To achieve it, the general idea is to use \textit{time weighted estimates} to glean some regularity
on $v_t$. Then, we are able to transfer the time regularity to space regularity thanks to elliptic
estimates and functional embeddings. The main idea is to differentiate the momentum equation (\ref{INS})$_2$ with respect to time and multiply it by $tv_t$. Next, using Gagliardo-Nirenberg-Sobolev inequality to handle the trouble-making  terms, one may exhibit the following estimates: 
\begin{align}
    \label{:t:w:e}
    \intT\ro t|v_t|^2+\int^T_0\left(\intT t|\n v_t|^2\right)dt\le C_{0,T},
\end{align}
where $C_{0,T}$ depending only on $\normep{\ro_0}$, $\normede{\sqrt{\ro}v}$, $\normede{\n v_0}$ and $T$.
Consequently, by Sobolev embedding and weighted Poincaré inequality \eqref{Poincare:l:h}, 
\begin{align}
    \label{conse:e:w:e:1}
      \lVert \sqrt{t} v_t\rVert_{L^2(0,T;L^q)} \le C_{0,T,q}, \esp{for all} q<\infty.
\end{align}
Another consequence of \eqref{:t:w:e} is that we have some control on the regularity of $v$ with respect to the time variable. This is given by the following estimate:
\begin{align}
    \label{conse:e:w:e:2}
    \lVert v \rVert_{H^{\frac{1}{2}-\alpha}(0,T;L^q)} \le C_{0,T,q}, \esp{for all} q<\infty,\; \alpha\in \left(0,\frac{1}{2}\right).
\end{align}
Then we also have a third consequence of estimate \eqref{:t:w:e}. Indeed, from the classical maximal regularity properties of the following  Stokes system, as in \cite{DanMUch13} (or in \cite{Mucha01,MuchaWoZaj02} in the context of the compressible Stokes system), 
\begin{align}
    \label{stokes:system}
    \begin{cases}
        -\Delta\sqrt{t}v +\n \sqrt{t} P=-\ro \sqrt{t}v_t-\sqrt{t}\ro v\cdot \n v \esp{in} \T^2,\\
        \dvg{\sqrt{t}v}=0 \esp{in} \T^2.
    \end{cases}
\end{align}
and \eqref{:t:w:e} one derives for all $0<\varepsilon<1$ and $2\le \eta\le \infty$ that
\begin{equation}
\label{est:Pres:n2v:SS}
\|\nabla^2\sqrt{t}v\|_{L^\eta(0,T;L^{\lambda-\varepsilon})} + \|\nabla\sqrt{t}P\|_{L^\eta(0,T;L^{\lambda-\varepsilon})} \leq C_{0,T,\varepsilon} \quad \text{with} \quad \lambda = \frac{2\eta}{\eta-2}. \end{equation}

Finally, having \eqref{est:Pres:n2v:SS} and $\n v$ in $L^\infty(L^2)$ ensures that $\sqrt{t}\n v $ is bounded in $ L^{2+\varepsilon}_T(L^\infty) $ for small enough $\varepsilon>0$, and thus, by H\"{o}lder inequality
\begin{align*}
   \lVert  \nabla v\rVert_{L^1_T(L^\infty)}=\int_0^T\normeinf{ \sqrt{t}\n v}\frac{dt}{\sqrt{t}}\le C_\varepsilon T^{\frac{\varepsilon}{2(2+\varepsilon)}} \lVert \sqrt{t}\n v\rVert_{ L^{2+\varepsilon}(L^\infty)  }.   
\end{align*}

The final step focuses on proving the uniqueness of the system (\ref{INS}). The estimate \eqref{n2:v:inLin:L2}
 allows us to reformulate the system (\ref{INS}) in Lagrangian coordinates without imposing higher regularity requirements on the data than those needed for existence. As highlighted in \cite{DanMu}, the hyperbolic nature of the mass equation results in the loss of one derivative for the density, preventing the application of a direct method based on stability estimates for (\ref{INS}). This derivative loss for the density also leads to a corresponding loss for the velocity. Since $\ro$ lacks sufficient regularity, any derivative loss is intolerable. To illustrate, if $(\ro_1, v_1)$ and $(\ro_2, v_2)$ are two solutions of (\ref{INS}) originating from the same initial data, and we denote by $(\delta \ro, \delta v)$ the difference between these two solutions, then we obtain from (\ref{INS})$_1$
%
\begin{align}
    \label{eq:del:ro}
    \frac{d}{dt}\delta \ro + v^2\cdot\n \delta \ro= -\underbrace{\delta v}_{L^2}\cdot \overbrace{\n \ro^1}^{\Dot{W}^{-1,p}}.
\end{align}
Therefore, it is not immediately clear whether uniqueness can be proven using this formulation. To overcome this challenge, as done in \cite{DanMu}, we rewrite the system (\ref{INS}) in \textit{Lagrangian coordinates}. In this framework, the loss of derivative does not occur when comparing two solutions of (\ref{INS}) originating from the same initial data. An additional benefit of using Lagrangian coordinates is that the density remains constant along the flow. Estimating the difference between solutions can be approached through basic energy arguments. The main challenge lies in the fact that the divergence is no longer zero, requiring the resolution of a "twisted" divergence equation to eliminate the non-divergence-free component. By applying the Gronwall's lemma, we establish uniqueness over a sufficiently small time interval, and by induction, extend this uniqueness over the entire existence time interval.
\\

\begin{rmq} 

For simplicity, we focus on the case $\Omega = \mathbb{T}^2$ in the following sections. A similar analysis applies to bounded domains with no-slip boundary conditions, but we omit those details here to enhance readability.

\end{rmq}

\section{Sobolev regularity }
\label{Sobolev:regu:sec}

The main goal of this section is to prove the following a priori estimates:
\begin{prop}
\label{pro:sobo:est}
Let $(\ro,v)$ a regular solution of (\ref{INS}) on $ [0,T]\times\T^2$ that satisfies 
\eqref{data:ass}.
Then there exists a positive constant $c_0$ depending only on the initial data such that for all $t\in[0,T]$
\begin{align}
    \label{sob:ine:p}
    \begin{split}
\sup_{\tau \in (0,t)}  \intT|\n v|^2(\tau) +\frac{1}{2}\int^t_0\intT\ro|v_t|^2 
    \le  c_0.
    \end{split}
\end{align}
Moreover, for any $r<2$ we have
\begin{equation}\label{est_nabla2}
    \int^t_0\normer{\n^2 v}^2+\int^t_0\normer{\n P}^2 \leq C(\|\rho_0\|_p)
\end{equation}
where 
$p=\frac{r}{2-r}$ and $C(\|\rho_0\|_p)\to\infty$ as $\|\rho_0\|_p\to\infty$.
\end{prop}

 \begin{rmq}
    \label{rmq:conse:Poi:Sobo:reg}
  From Proposition \ref{pro:sobo:est} and Poincar\'e inequality (\ref{Poincare:l:h}) it follows that for all $q<\infty$
    \[ \|v\|_{L^\infty(0,T;L^q)}\leq C_0, \]
    where $C_0$ depends only on the initial data. Using additionally the fact that $\|\ro\|_p=\|\ro_0\|_p<\infty$ for all $p<\infty$, we further obtain for all $0\leq \alpha<\infty$, $1\leq q <\infty$
    \begin{equation}\label{normq:al:v} \|\ro^\alpha v\|_{L^\infty(0,T;L^q)}\leq C_0, \end{equation}
    where again $C_0$ depends only on $(\ro_0,v_0)$. \\
    In the case when $\ro_0\in L^p(\Omega)$ for fixed $p$, the estimate (\ref{normq:al:v}) holds for any $ 0\le \alpha <p $ and $1\le q < \frac{p}{\alpha}$ instead.
    
 \end{rmq}

\begin{proof}[Proof of Proposition \ref{pro:sobo:est}]
As in \cite{DanMu} the proof consists in performing an energy method and introducing a suitable 
 'energy' functional that contains $H^1$ information on the velocity. Note that in contrast with \cite{DanMu}, the density is not necessary bounded, therefore in order to succeed in getting a time-independent control on the solution in terms of
the data, we will  need to combine an interpolation inequality and condition \eqref{new_initial}.  

    Testing the momentum equation of  (\ref{INS}) by $v_t$ yields
    \begin{align*}
       & \frac{1}{2}\frac{d}{dt}\intT|\n v|^2+\intT\ro|v_t|^2 = -\intT (\ro v\cdot\n v)\cdot v_t\\
        &\le \frac{1}{2} \intT\ro|v_t|^2 +\frac{1}{2}\intT\ro |v\cdot\n v|^2.
    \end{align*}
So we deduce that 
\begin{align}
    \label{1:est:nabla v:}
   \frac{d}{dt} \intT|\n v|^2+\intT\ro|v_t|^2 \le  \intT\ro |v\cdot\n v|^2.
\end{align}
From the momentum equation, we can rewrite the second derivative and the gradient of the pressure as follows: 
\begin{align*}
    -\Delta v+\n P=-\ro v_t-\ro v\cdot\n v.
\end{align*}
In conclusion, from the maximal regularity of  (\ref{INS}) on the torus $\T^2$, we have for $1< r<2$
\begin{align*}
   \normer{\n^2 v}^2+\normer{\n P}^2 &\le C_r\normer{\sqrt{\ro}(\sqrt{\ro}v_t+\sqrt{\ro}v\cdot\n v)}^2.
\end{align*}
Now, thanks to Hölder's inequality, we get
\begin{align}
 \label{2:est:nabla v:}
   \normer{\n^2 v}^2+\normer{\n P}^2 \le  C_r\normep{\ro}(\normede{\sqrt{\ro}v_t}^2+\normede{\sqrt{\ro}v\cdot\n v}^2)
\end{align}
with 
    $\frac{1}{r}= \frac{1}{2}+\frac{1}{2p}.$
Combining \eqref{1:est:nabla v:} and \eqref{2:est:nabla v:}, we deduce that 
\begin{align}
 \label{3:est:nabla v:}
    \frac{d}{dt}\intT|\n v|^2+\frac{1}{2}\intT\ro|v_t|^2+ \frac{1}{2C_r\normep{\ro}}\left(\normer{\n^2 v}^2+\normer{\n P}^2 \right)\le  \frac{3}{2}\intT\ro |v\cdot\n v|^2.
\end{align}
To handle the right- hand side of \eqref{3:est:nabla v:}, we use Hölder's inequality and Sobolev inequality that states
\begin{equation}
    \label{sobolev:ine:criti}
    \lVert\n v\rVert_{2p}\le C  \lVert\n^2 v\rVert_{r}.
\end{equation}
We deduce that:
\begin{align*}
   \intT\ro |v\cdot\n v|^2 &\le \normede{\sqrt{\ro}|v|^2}\normes{\ro^\frac{1}{4}\n v}{4}^2\\
   &\le  \normep{\ro}^{\frac{1}{2}} \normede{\sqrt{\ro}|v|^2} \normes{\n v}{q_4}^2\\
   &\le    \normep{\ro}^{\frac{1}{2}} \normede{\sqrt{\ro}|v|^2} \normede{\n v}^{2(1-\alpha)}\normes{\n^2 v}{r}^{2\alpha}.
\end{align*}
and $q_4\in (4,2p)$ and $\alpha\in (0,1)$ verify
\begin{equation}
    \label{alpha:q_4}
    \frac{1}{4}=\frac{1}{4p}+\frac{1}{q_4} \esp{and } \frac{1}{q_4}=\frac{\alpha}{2p}+\frac{1-\alpha}{2}.
\end{equation}
Let us remark that \eqref{alpha:q_4} provides 
\begin{equation}
    \label{cnd:p:alpha:q4}
    \frac{1}{2-2\alpha}=\dfrac{1-\frac{1}{p}}{1-\frac{3}{p}} \esp{and } p>3.
\end{equation}

Hence, using Young inequality yields 
\begin{align}
    \label{4:est:nabla v:}
     \intT\ro |v\cdot\n v|^2 &\le \frac{1}{4C_r\normep{\ro}} \normes{\n^2 v}{r}^{2} + C_{r,p} \normede{\n v}^2  \normep{\ro}^{\frac{1+2\alpha}{2(1-\alpha)}} \normede{\sqrt{\ro}|v|^2}^\frac{1}{1-\alpha}
\end{align}
with 
\begin{equation}
    \label{def C:r:p}
        C_{r,p}\overset{def}{=} (1-\alpha)\Big(\frac{1}{4\alpha}C_r \Big)^{-\frac{\alpha}{1-\alpha}}.
\end{equation}
Note that since $\alpha\to \frac{1}{2}$ when $p\to \infty$, the constant $C_{r,p}$ can be bounded independently of $p$. Using the lemma \ref{ladyn:log:cor:lem}, we are able to bound the term $  \normede{\sqrt{\ro}|v|^2}$. Remarking that the function $z\mapsto z\log(e+\frac{A}{z})$ (with $A>0$) is increasing over $[0,\infty)$, and remembering \eqref{energy:balance}, \eqref{conser:BMO:Lp} \eqref{conser:momen}, we have
\begin{align}
   \label{:est:nabla v:5}
   \normede{\sqrt{\ro}|v|^2}&\le C \normede{\sqrt{\ro_0}v_0} \left( \intT\ro_0v_0+C_p \normede{\n v} \log^{\frac{1}{2}} 
   \left( e+\frac{\normede{\ro-M}^2}{M^2}+\normep{\ro_0}\frac{\normede{\n v}^2}{\normede{\sqrt{\ro_0}v_0}^2}\right)
   \right)\nonumber
   \\
   &\le  C \normede{\sqrt{\ro_0}v_0} \left( e+ \normede{\n v} \right) \log^{\frac{1}{2}}
   \left( e+\frac{\normede{\ro-M}^2}{M^2}+\normep{\ro_0}\normede{\sqrt{\ro_0}v_0}^2\right)
   \log^{\frac{1}{2}}\left(e+\normede{\n v}^2\right)
\end{align}
Hence, inequality \eqref{4:est:nabla v:} becomes
\begin{align}
    \label{:est:nabla v:6}
     \intT\ro |v\cdot\n v|^2 &\le \frac{1}{4C_r\normep{\ro}} \normes{\n^2 v}{r}^{2} + K_0\normede{\n v}^2\left( e+ \normede{\n v}^2 \right)^{\frac{1}{2(1-\alpha)}}\log^{\frac{1}{2(1-\alpha)}}\left(e+\normede{\n v}^2\right)
\end{align}
with $K_0=: C_p \normep{\ro_0}^{\frac{1+2\alpha}{2(1-\alpha)}}\normede{\sqrt{\ro_0}v_0}^{\frac{1}{1-\alpha}} \log^{\frac{1}{2(1-\alpha)}}\left( e+\frac{\normede{\ro-M}^2}{M^2}+\normep{\ro_0}\normede{\sqrt{\ro_0}v_0}^2\right). $

Plugging (\ref{:est:nabla v:6}) into (\ref{3:est:nabla v:}), we get that in particular
\[ \frac{d}{dt}\int_{\T^2}|\nabla v|^2 + \frac{1}{2}\int_{\T^2}\rho|v|^2 \leq f(t)\left( e+ \normede{\n v}^2 \right)^{\frac{1}{2(1-\alpha)}}\log^{\frac{1}{2(1-\alpha)}}\left(e+\normede{\n v}^2\right),
\]
where $f(t):= \normede{\n v}^2 K_0$. Denoting
\begin{equation}\label{def:X:K_0}
    X(t) := e+ \intT|\n v|^2+\frac{1}{2}\int^t_0\intT\ro|v_t|^2,
\end{equation}
we end up with

\begin{align}
   \label{:est:nabla v:7}
   \frac{d}{dt}X(t) \leq f(t)(X\log X)^{1+s},
\end{align}
where $s=\frac{1}{2(1-\alpha)}-1=\frac{2}{p-3}$. Solving this ODE, we get
\begin{equation}\label{int_ineq} \int_{X_0}^{X(t)}\frac{dx}{(x\log x)^{1+s}} \leq \int_0^t f(\tau)d\tau. \end{equation}
Now we proceed by contradiction: assuming that the solution $X$ cannot be extended for the whole interval $[0,T]$, there has to exist some $T^*<T$ such that $X(t)\to\infty$ as $t\to T^*$. Then we would have
\[ \int_{X_0}^\infty \frac{\dd x}{(x\log x)^{1+s}} \leq \int_0^{T^*} f(\tau)\dd\tau. \]
Therefore the sufficient condition to have a global solution at the interval $[0,T]$, i. e. for the blow-up to not occur, is 
\begin{align}
\label{cnd:base}
    \int_0^T f(\tau)d\tau < \int_{X_0}^\infty \frac{dx}{(x\log x)^{1+s}} \esp{for } s=\frac{2}{p-3}.
\end{align}

Since $f(t)=\|\nabla v\|_2^2K_0$, and \eqref{energy:balance} provides  $ \displaystyle\int_0^T\|\nabla v\|_2^2 d\tau \leq \|\sqrt{\rho_0}v_0\|_2^2$, we can replace this condition by
\begin{equation}\label{new_condition} K_0\|\sqrt{\rho_0}v_0\|_2^2 < \int_{X_0}^\infty \frac{dx}{(x\log x)^{1+s}}\esp{for } s=\frac{2}{p-3}. \end{equation}
For a fixed $s$, the integral on the right- hand side of (\ref{new_condition}) is finite, therefore for large enough data the blow-up may occur in finite time. On the other hand, the definition of $K_0$, condition (\ref{new_initial}) and the fact that $s\sim\frac{2}{p}$ (when $ p\to \infty$) provide that
\[ \lim_{s\to 0}\frac{1}{|\log s|} K_0\|\sqrt{\ro_0}v_0\|_2^2=0 . \] 
Therefore in order for (\ref{new_condition}) to be satisfied for large $p$, it is enough to show that
\[ \lim_{s\to 0} \frac{1}{|\log s|}\int_{X_0}^\infty \frac{dx}{(x\log x)^{1+s}} \geq l \]
for some  $l>0$. This is however provided by Lemma \ref{lem:basi:lim} below, with $l=1$. In conclusion, condition (\ref{new_condition}) is fulfilled for some small enough $s_0>0$, which, since this condition does not depend on $t$, completes the proof of (\ref{sob:ine:p}).

Now, to prove (\ref{est_nabla2}), we revert back to (\ref{3:est:nabla v:}). Integrating (\ref{3:est:nabla v:}) in time and using (\ref{:est:nabla v:6}), we get
\begin{multline*} \frac{1}{4C_r\|\rho_0\|_p}\int_0^t \|\nabla^2 v\|_r^2 + \frac{1}{2C_r\|\rho_0\|_p}\int_0^t\|\nabla P\|_r^2 \\
\leq \|\nabla v_0\|_2^2 + K_0\|\sqrt{\rho_0}v_0\|_2^2 \left( e+ \normede{\n v(t,\cdot)}^2 \right)^{1+s}\log^{1+s}\left(e+\normede{\n v(t,\cdot)}^2\right). \end{multline*}
Since from (\ref{sob:ine:p}) we have $\|\nabla v\|_{L^\infty(0,T;L^2)}^2\leq c_0$ and $s\sim\frac{2}{p}$, we end up with
\begin{align*}
    \int_0^t\|\nabla^2 v\|_r^2 + \int_0^t\|\nabla P\|_r^2 &\leq C\|\rho_0\|_p\biggl(\|\nabla v_0\|_2^2 + K_0\biggl((c_0+e)\log (e+c_0)\biggr)^{1+\frac{2}{p}}\biggr)\\ 
    &\sim \|\rho_0\|_p^3\log(e+\|\rho_0\|_p), 
\end{align*}
for $\|\rho_0\|_p\to \infty$, which ends the proof of (\ref{est_nabla2}).
\end{proof}
The following result allows us to clarify the condition \eqref{new_condition}:
\begin{lem}
\label{lem:basi:lim}
It holds
    \[ \lim_{s\to 0} \frac{1}{|\log s|} \int_{X_0}^\infty \frac{dx}{(x\log x)^{1+s}} \geq 1. \]
\end{lem}
\begin{proof}
For a fixed $x\geq X_0$, let
\[ H_x(s) = \frac{1}{(x\log x)^{1+s}}, \quad s\in [0,1]. \]

     It is easy to check that $s\mapsto H_x(s)$ is convex and the derivative of $H_x$ satisfies
 \begin{align*}
     H'_x(s)= -\frac{1}{x^{1+s}\log(x)^{s}}- \frac{\log(\log(x))}{x^{1+s}\log(x)^{1+s}}.
 \end{align*}
Using convexity, we have
\[ H_x(s) \geq H_x(0) + sH_x'(0) = \frac{1}{x\log x} - s\left(\frac{1}{x} + \frac{\log\log x}{x\log x}\right). \] 
Next, for $A$ large enough we deduce that
\[\begin{aligned} \int_{X_0}^A H_x(s) \;\dd x &\geq \int_{X_0}^A \frac{1}{x\log x}\;\dd x - s\int_{X_0}^A \frac{1}{x}+\frac{\log\log x}{x\log x} \;\dd x \\
&= \log\log A - \log\log X_0 - s(\log A-\log X_0) - \frac{s}{2}\left((\log\log A)^2-(\log\log X_0)^2\right).
\end{aligned}\]
%
 In particular, choosing $ \log(A) =\frac{1}{s}  $ (that is $A=e^\frac{1}{s}$), we obtain  
 \begin{equation}\label{int_H_x}
    \int^{e^\frac{1}{s}}_{X_0} H_x(s)dx \ge -\log(s)-1-\frac{s}{2}\log^2(s)+C(X_0),
\end{equation}
where $C(X_0)$ is independent on $s$. Dividing both sides by $-\log s$ and passing to the limit with $s\to 0$, we finish the proof of Lemma \ref{lem:basi:lim}.

It’s important to note that a direct computation of the integral from the lemma reveals that the asymptotic behavior of this quantity is $-\log s$ as $s \to 0$. However, we opted for a proof based on the properties of convex functions to highlight the strength and effectiveness of this technique.
\end{proof}

\begin{rmq}
With the help of Lemma \ref{lem:basi:lim}, we can also express the constant $c_0$ in (\ref{sob:ine:p}) in terms of $s_0$: let $s_0$ be such that
\[ \frac{1}{|\log s_0|}\left(1+\frac{s_0}{2}\log^2(s_0) -C(X_0)\right) + \frac{1}{|\log s_0|}K_0\|\sqrt{\rho_0}v_0\|_2^2 <\frac{1}{2}, \]
when $C(X_0)$ is like in (\ref{int_H_x}) (such $s_0$ exists, since $\lim_{s\to 0}\frac{1}{|\log s|}K_0=0$). Assuming that there exists a $t_*$ such that $X(t_*)>e^{\frac{1}{s_0}}$, from (\ref{int_ineq}) we have
\[ \int_{X_0}^{e^{\frac{1}{s_0}}}\frac{dx}{(x\log x)^{1+s_0}} < \int_{X_0}^{X(t_*)}\frac{dx}{(x\log x)^{1+s_0}} \leq \int_0^{t_*}f(\tau)d\tau \leq K_0\|\sqrt{\rho_0}v_0\|_2^2. \]
Using (\ref{int_H_x}) and dividing by $|\log s_0|$, we get
\[ 1 - \frac{1}{|\log s_0|}\left(1+\frac{s}{2}\log^2(s_0) -C(X_0)\right) \leq \frac{1}{|\log s_0|}K_0\|\sqrt{\rho_0}v_0\|_2^2 \]
which leads to a contradiction. In conclusion, $X(t)\leq e^{\frac{1}{s_0}}$ for all $t>0$.
\end{rmq}

   

In the absence of vacuum, Proposition \ref{pro:sobo:est} remains valid, provided that the density also belongs to the space $\mathcal{Y}_0$. More precisely,
\begin{prop}
\label{pro:sobo:est:r0>0}
Let $(\ro, v)$ be a smooth solution to (\ref{INS}) on $ [0,T] $ satisfying \eqref{data:ass:ro>0}. Then the estimates \eqref{sob:ine:p} and \eqref{est_nabla2} hold true for all $t\in [0,T]$, with $c_0$ depending on initial data. 
\end{prop}
\begin{proof}
   The main ingredient of the proof relies on the fact that $v\in L^\infty(\mathbb{R}_+, \T^2)$.\\
   In fact, since  $\ro$ is far away from vacuum, the identity \eqref{energy:balance} provides  
  \begin{align}
  \label{normde(v)=f(normde(ro:v))}
      \normede{v}\le \frac{1}{\sqrt{\ro_*}}\normede{\sqrt{\ro}v}\le  \frac{1}{\sqrt{\ro_*}}\normede{\sqrt{\ro_0}v_0}.
  \end{align}
  Starting from \eqref{3:est:nabla v:} and using Hölder inequality, we have
    \begin{align*}
        \intT \ro |v\cdot\n v|^2\le \normep{\ro}\normes{v}{2p}^2\normes{\n v}{\frac{2p}{p-2}}^2
    \end{align*}
    from which and Sobolev inequality we deduce
      \begin{align*}
        \intT \ro |v\cdot\n v|^2\le \normep{\ro}\normes{v}{2}^{\frac{2}{p}}\normede{\n v}^{2(1-\frac{1}{p})}\normede{\n v}^{2\frac{p-3}{p-1}} \normes{\n^2 v}{r}^{\frac{4}{p-1}}.
    \end{align*}
    Therefore, Young inequality ensures that, for all $\varepsilon>0$ 
     \begin{align*}
        \intT \ro |v\cdot\n v|^2\le  \frac{\varepsilon}{\normep{\ro}}\normes{\n^2 v}{r}^2+ \varepsilon^{-\frac{2}{p-3}} 
        \normep{\ro}^{\frac{p+1}{p-3}}   \normes{v}{2}^{2\frac{p-1}{p(p-3)}}\normede{\n v}^{2(2+\varepsilon_p)} ,
    \end{align*}
  where   
$\varepsilon_p:=\frac{p+1}{p(p-3)}$  \\
  Combining the previous inequality and \eqref{3:est:nabla v:}, choosing $\varepsilon$ small enough and keeping the notation in the proof of Proposition \ref{pro:sobo:est} we end up with
    \begin{align*}
        \frac{d}{dt} X(t)\le g(t)X^{1+\varepsilon_p} \esp{with } g(t):=C  \normep{\ro}^{\frac{p+1}{p-3}}   \normes{v}{2}^{2\frac{p-1}{p(p-3)}}\normede{\n v}^2.
    \end{align*}
The sufficient condition for getting the existence of $X(t)$ for all $t\in \mathbb{R}^+$ is to find $p$ such that
\begin{align}
\label{cnd:X:g}
    \int^t_0 g(\tau)d\tau < \int^{\infty}_{X_0} x^{-1-\varepsilon_p}dx=\frac{1}{\varepsilon_p}X_0^{-\varepsilon_p}.
\end{align}
Since $ \int^t_0 \normede{\n v}^2\le \normede{\sqrt{\ro_0}v_0}^2 $, then from \eqref{normde(v)=f(normde(ro:v))} we have  
 $$  \int^t_0 g(\tau)d\tau\le C \ro_*^{-\frac{p-1}{p(p-3)}}  \normep{\ro_0}^{\frac{p+1}{p-3}}   \normes{\sqrt{\ro_0}v_0}{2}^{2(1+\frac{p-1}{p(p-3)})}  .$$
Therefore  taking advantage of the fact that $\ro\in \mathcal{Y}_0$ , there exists $p_0$ such that  \eqref{cnd:X:g} holds. This completes the proof of the Proposition \ref{pro:sobo:est:r0>0}.
%
\end{proof}

\section{Weighted estimates}
\label{sec:weighted:estimates}
Our goal in this section is to prove that $\n v$ belongs to $L^1(0,T;L^\infty)$, in terms of the data and the time $T$. To achieve it, let us first prove we have a bound on $\sqrt{\ro t}v_t$ and $\sqrt{t}\n v_t$ in $ L^\infty(0,T;L^2) $ and $L^2(0,T;L^2)$ respectively. 

The below estimation is based on the scheme form \cite{DanMu}. We put here all the details since the lack of boundedness of the density changes the key estimates in a few important places. We perform the calculations assuming that $\ro_0\in L^p(\Omega)$ with $p>3$ fixed. In particular, the obtained estimates will be still valid provided that $\ro_0\in\mathcal{L}$ (or $ \ro_0\in \mathcal{Y}_0 $, in the context of Theorem \ref{thm:Lp:ro>0}).

\begin{prop}(Time derivative estimates).
\label{propo:TIme:derivative}
Let $(\ro, v)$ be a smooth enough solution to system (\ref{INS}) on $[0,T^*)\times \T^2$ satisfying  \eqref{sob:ine:p} and \eqref{est_nabla2}. Then for all $T\in [0,T^*) $ , it holds

\begin{align}
\label{ine:time:deriv}
    \sup_{t\in [0,T]} \normede{\sqrt{t\ro(t)} v_t(t)}^2+\int^T_0\normede{\sqrt{t}\n v_t(t)}^2dt\le C_{0,T},
\end{align}
\begin{align}
\label{ine:time:deriv:-ro}
   \lVert \sqrt{t} v_t\rVert_{L^2(0,T;L^q)} \le C_{0,T,q}, \esp{for all} q<\infty
\end{align}
    where $C_{0,T}$ is a constant depending only on $T$, $p$ and the norms $\normep{\ro_0}$, $\normede{\sqrt{\ro_0}v_0}$, $\normede{\n v_0}$, while $C_{0,T,q}>0$ depends only on $T$, $p,q$ and the norms $\normep{\ro_0}$, $\normede{\sqrt{\ro_0}v_0}$, $\normede{\n v_0}$.
\end{prop}
\begin{proof}
    
At first, differentiating $(\ref{INS})_2$, yields
\begin{equation}
    \label{diff:time:es:1}
    \ro v_{tt} +\ro_t v_t+\ro_t v\cdot \n v+\ro v_t\cdot\n v+\ro v_t\cdot\n v-\Delta v_t +\n P_t=0.
\end{equation}
Next, taking the scalar product with $tv_t$, we get
\begin{equation}
    \label{diff:time:es:2}
    \frac{1}{2}\frac{d}{dt} \intT\ro t|v_t|^2+\intT t|\n v_t|^2=\sum_{i=1}^4 J_i
\end{equation}
with
\begin{align}
    \label{J_1}
    &J_1=\frac{1}{2}\intT\ro |v_t|^2,\\
    \label{J_2}
    & J_2=-\intT\left(t\ro_t |v_t|^2+(\sqrt{t}\ro v\cdot\n v_t)\cdot(\sqrt{t}v_t)\right),\\
    \label{J_3}
    &J_3=-\intT\left(\sqrt{t}\ro_tv\cdot\n v\right)\cdot\left(\sqrt{t}v_t\right),\\
    \label{J_4}
    &J_4=-\intT \left(\sqrt{t}\ro v_t\cdot\n v\right)\cdot\left(\sqrt{t}v_t\right).
\end{align}
Now, we are going to show that for all $t\in [0,T]$ :
\begin{equation}
     \label{diff:time:es:3}
     \sum_{i=1}^4 J_i \le \frac{1}{2}\normede{\sqrt{t}\n v_t}^2+C(1+\normede{\sqrt{t\ro}v_t}^2)h(t).
\end{equation}
for some $h\in L^1(0,T)$, the norm of which may depend on $T$ and the initial data.

Indeed, having \eqref{diff:time:es:3} at hand will enable us to get \eqref{ine:time:deriv} by means of Gronwall's Lemma, since the first term in the right-hand side of \eqref{diff:time:es:3} may be absorbed by the left-hand of \eqref{diff:time:es:2}.

 Obviously, according to  estimate \eqref{sob:ine:p} we have $ J_1\in L^1(0,T)$, and its norm depends only on the initial data .\\

 To handle $J_2$, we use the mass equation $(\ref{INS})_1$ in the first term, after performing an integration by parts. One has
\begin{align*}
     J_2&=-\intT\left(t\ro_t |v_t|^2+(\sqrt{t}\ro v\cdot\n v_t)\cdot(\sqrt{t}v_t)\right)\\
    &= \intT  \left(t\ro v \n|v_t|^2-(\sqrt{t}\ro v\cdot\n v_t)\cdot(\sqrt{t}v_t)\right)\\
    &\le 3\intT \ro |v||\sqrt{t}\n v_t||\sqrt{t}v_t|.
\end{align*}
Using Hölder and interpolation inequalities, we get
\begin{align*}
    J_2&\le 3\normede{\sqrt{t}\n v_t}\normes{\sqrt{t\ro}v_t}{q_1}\normes{\sqrt{\ro}v}{q_2}\\
    &\le 3\normede{\sqrt{t}\n v_t}\normede{\sqrt{t\ro}v_t}^{1-\theta}\normes{\sqrt{t\ro}v_t}{q}^{\theta} \normes{\sqrt{\ro}v}{q_2}
\end{align*}
with
\begin{equation}
    \label{alpha=p2:q2:qq2}
  \frac{1}{2}=  \frac{1}{q_1}+\frac{1}{q_2},\esp{} q_2<2p,\esp{} q_1<q<2p,  \esp{} \frac{1}{q_1}=\frac{1-\theta}{2}+\frac{\theta}{q}, \text{   and  } \theta\in (0,1)
\end{equation}
For instance we can take $q_2=q=6, \theta=\frac{1}{2}$.\\
Using Hölder's inequality on $\normes{\sqrt{t\ro}v_t}{q} $ gives
\[     \normes{\sqrt{t\ro}v_t}{q} \le \normes{\sqrt{\ro}}{2p}\normes{\sqrt{t}v_t}{m},   \]
with $\frac{1}{m}:=\frac{1}{q}-\frac{1}{2p}$. In the end, we arrive at
\begin{align*}
    J_2&\le 3 \normes{\ro}{p}^{\frac{\theta}{2}} \normede{\sqrt{t}\n v_t}\normede{\sqrt{t\ro}v_t}^{1-\theta}\normes{\sqrt{t}v_t}{m}^\theta\normes{\sqrt{\ro}v}{q_2}
\end{align*}
which becomes, using \eqref{Poincare:l:h} and Young inequality 
\begin{align*}
    J_2&\le C \normes{\ro}{p}^{\frac{\theta}{2}}  \normede{\sqrt{t}\n v_t}\normede{\sqrt{t\ro}v_t} \normes{\sqrt{\ro}v}{q_2}\\
    &+C \normes{\ro}{p}^{\frac{\theta}{2}}  \normede{\sqrt{t}\n v_t}^{1+\theta}\normede{\sqrt{t\ro}v_t}^{1-\theta} \normes{\sqrt{\ro}v}{q_2}\\
    &\le \frac{1}{8} \normede{\sqrt{t}\n v_t}^{2}+C \normes{\ro}{p}^{\theta} \normede{\sqrt{t\ro}v_t}^2 \normes{\sqrt{\ro}v}{q_2}^2\\
    & +C\normes{\ro}{p}^{\frac{\theta}{1-\theta}}  \normede{\sqrt{t\ro}v_t}^2 \normes{\sqrt{\ro}v}{q_2}^{\frac{2}{1-\theta}}.
\end{align*}
Finally, $J_2$ may be bounded as follows 
\begin{align}
    \label{bound:J2}
    J_2\le  \frac{1}{8} \normede{\sqrt{t}\n v_t}^{2} +h_2(t)\normede{\sqrt{t\ro}v_t}^2
\end{align}
 where
 \[h_2= C \left(\normes{\ro}{p}^{\theta}  \normes{\sqrt{\ro}v}{q_2}^2 +\normes{\ro}{p}^{\frac{\theta}{1-\theta}} \normes{\sqrt{\ro}v}{q_2}^{\frac{2}{1-\theta}} \right).\]
Notice that, since \eqref{sob:ine:p} and \eqref{est_nabla2} are assumed to be satisfied ,  Remark \ref{rmq:conse:Poi:Sobo:reg} and $q_2< 2p$ ensure that $ \normes{\sqrt{\ro}v}{q_2}\in L^\infty(0,T)$. Hence $h_2\in L^1(0,T)$ and its norm depends on $T$ and the initial data.

Next, from the mass equation, we get
\begin{align*}
    J_3 &=-\intT\left(\sqrt{t}\ro_tv\cdot\n v\right)\cdot\left(\sqrt{t}v_t\right)
    =-\intT t\ro v\cdot \n \left( (v\cdot\n v)\cdot v_t\right),
\end{align*}
where we have performed an integration by parts in the second equality.\\
Hence,
\begin{align*}
    J_3\le \underbrace{\intT t\ro |v||v_t||\n v|^2}_{J_{31}}+\underbrace{\intT t\ro |v|^2|v_t||\n^2 v|}_{J_{32}}+ \underbrace{\intT t\ro |v|^2|\n v_t||\n v|}_{J_{33}}.
\end{align*}
To handle $J_{31}$, we use Hölder inequality and \eqref{sobolev:ine:criti}. One has
\begin{align*}
   J_{31} &=  \intT t\ro |v||v_t||\n v|^2\\
   & \le \normede{\sqrt{t\ro}v_t}\normes{\n v}{2p}^2\normes{\sqrt{t\ro}v}{q} \\
   &\le C(1+\normede{\sqrt{t\ro}v_t}^2)\normes{\n^2 v}{r}^2\normes{\sqrt{t\ro}v}{q}
\end{align*}
with 
   $\frac{1}{q}:= \frac{1}{2}- \frac{1}{p} > \frac{1}{2p}.$
Then, Remark \ref{rmq:conse:Poi:Sobo:reg} ensures that  $\normes{\sqrt{t\ro}v}{q} \in L^\infty(0,T)$ and thus, we get 
\begin{align}
    \label{bound:J31}
    J_{31}\le  h_{31}(t)(1+\normede{\sqrt{t\ro}v_t}^2)
\end{align}
with 
\[ h_{31}=\normes{\n^2 v}{r}^2\normes{\sqrt{t\ro}v}{\bar{p}}\in L^1(0,T), \]
and a norm depending on $T$ and initial data.

To bound $ J_{32}$, we take advantage on Hölder inequality to get
\begin{align*}
    J_{32}&= \intT t\ro |v|^2|v_t||\n^2 v|\\
    &\le \normes{\sqrt{\ro}}{2p}\normes{\n^2 v}{r}\normes{\sqrt{t}v_t}{r_1}\normes{\sqrt{t\ro}|v|^2}{r_2},
\end{align*}
where $r$ has been defined in \eqref{2:est:nabla v:} and 
\[ \frac{1}{2p}+\frac{1}{r}+\frac{1}{r_1}+\frac{1}{r_2}=1. \]
Note that for $p>3$, if $r_1$ is sufficiently large we can pick $r_2<2p$.
Applying \eqref{Poincare:l:h} to $ \normes{\sqrt{t}v_t}{r_1} $, we  find constant $C>0$ depending on the data and $p $ such that
\begin{align*}
   J_{32} \le & C \normes{\ro}{p}^\frac{1}{2} \normes{\n^2 v}{r}\left(\normes{\sqrt{t\ro}v_t}{2}+\normede{\sqrt{t}\n v_t}\right)\normes{\sqrt{t\ro}|v|^2}{r_2}\\ 
   \le & \frac{1}{16}\normede{\sqrt{t }\n v_t}^2+ C\normes{\ro}{p}\normes{\n^2 v}{r}^{2} \normes{\sqrt{t\ro}|v|^2}{r_2}^{2}\\
   &+ C\normes{\ro}{p}^\frac{1}{2}  \normes{\n^2 v}{r} \normede{\sqrt{t\ro}v_t}\normes{\sqrt{t\ro}|v|^2}{r_2}.
\end{align*}
Using the inequality $  \normede{\sqrt{t\ro}v_t} \le C(1+\normede{\sqrt{t\ro}v_t}^2) $, we deduce that 
\begin{equation}
    \label{bond:J32}
     J_{32}\le \frac{1}{16}\normede{\sqrt{t }\n v_t}^2+ (1+\normede{\sqrt{t\ro}v_t}^2)h_{32},
\end{equation}
where
\begin{equation}
    \label{def=h32}
    h_{32}=C\normes{\ro}{p}\normes{\n^2 v}{r}^{2} \normes{\sqrt{t\ro}|v|^2}{r_2}^{2}+ C\normes{\ro}{p}^\frac{1}{2}  \normes{\n^2 v}{r} \normes{\sqrt{t\ro}|v|^2}{r_2}.
\end{equation}
 From  \eqref{est_nabla2}, we have $ \normes{\n^2v}{r}^2 \in L^1(0,T) $ and from Remark  \ref{rmq:conse:Poi:Sobo:reg}, we know that $ \normes{\sqrt{t\ro}|v|^2}{r_2} \in L^\infty(0,T) $. In consequence, $h_{32} \in L^1(0,T)$ and its norm depends  on $T$ and initial data.   

For $J_{33}$, we have
\begin{align*}
   J_{33} &=\intT t\ro |v|^2|\n v_t||\n v|\\
   &\le \sqrt{T}\normede{\sqrt{t}\n v_t}\normes{\sqrt{\ro}}{2p}\normes{\n v}{2p}\normes{\sqrt{\ro}|v|^2}{p^*}\\
   &\le \frac{1}{16}\normede{\sqrt{t}\n v_t}^2+ C T\normes{\ro}{p}\normes{\n^2 v}{r}^2\normes{\sqrt{\ro}|v|^2}{p^*}^2.
\end{align*}
with (since $p>3$)
\begin{align}
    \label{def:q:est:t}
   \frac{1}{p^*}:= \frac{1}{2}-\frac{1}{p} > \frac{1}{2p} .
\end{align}
Notice that  $\normes{\ro}{p}\in L^1(0,T) $ and according to  \eqref{est_nabla2}, $\normes{\n^2 v}{r}^2$ belongs to $L^1(0,T)$ as well.
 Moreover, as $p^*<2p$, it follows from Remark \ref{rmq:conse:Poi:Sobo:reg} that $\normes{\sqrt{\ro}|v|^2}{p^*}^2\in L^\infty(0,T)$. In consequence, we have

\begin{equation}
    \label{bond:J33}
     J_{33}\le \frac{1}{16}\normede{\sqrt{t }\n v_t}^2+ (1+\normede{\sqrt{t\ro}v_t}^2)h_{33}
\end{equation}
and $h_{33}:= C T\normes{\ro}{p}\normes{\n^2 v}{r}^2\normes{\sqrt{\ro}|v|^2}{p^*}^2 \in L^1(0,T)$ and its norm depends  on $T$ and initial data.

To bound $J_4$, we may write
\begin{align*}
    J_4=& -\intT \left(\sqrt{t}\ro v_t\cdot\n v\right)\cdot\sqrt{t}v_t \\
    &\le \normede{\sqrt{t\ro}v_t}\normes{\sqrt{\ro}}{2p}\normes{\sqrt{t}v_t}{p^*}\normes{\n v}{2p}
\end{align*}
with $\frac{1}{p^*}=\frac{1}{2}-\frac{1}{p}$.\\
Thanks to \eqref{Poincare:l:h}, we are able to write
\begin{align*}
    J_4 &\le C\normede{\sqrt{t\ro}v_t}\normes{\sqrt{\ro}}{2p}\left( \normede{\sqrt{t\ro}v_t}+ \normede{\sqrt{t}\n v_t} \right)\normes{\n v}{2p}\\
    &\le  C\normes{\sqrt{\ro}}{2p}\normede{\sqrt{t\ro}v_t}^2\normes{\n v}{2p}   + C\normes{\sqrt{\ro}}{2p}\normede{\sqrt{t\ro}v_t}\normede{\sqrt{t}\n v_t}\normes{\n v}{2p}\\
    &\le \frac{1}{8}\normede{\sqrt{t}\n v_t}^2+C \left(  \normes{\sqrt{\ro}}{2p}\normes{\n v}{2p}+  \normes{\sqrt{\ro}}{2p}\normes{\n v}{2p}^2\right)\normede{\sqrt{t\ro}v_t}^2.
\end{align*}
Using again that $ \n^2 v \in L^2(0,T;L^r) $ and (\ref{sobolev:ine:criti}), we obtain
\begin{equation}
    \label{bound:J4}
    J_4\le \frac{1}{8}\normede{\sqrt{t}\n v_t}^2+ h_4\normede{\sqrt{t\ro}v_t}^2,
\end{equation}
where 
\[ h_4:=C \left( \normes{\sqrt{\ro}}{2p}\normes{\n v}{2p}+  \normes{\sqrt{\ro}}{2p}\normes{\n v}{2p}^2\right) \in L^1(0,T) \]
and its norm depends  on $T$ and initial data. This completes the proof of inequality \eqref{diff:time:es:3}.  \\
Substituting  inequality  \eqref{diff:time:es:3} into identity \eqref{diff:time:es:2} and taking advantage on Gronwall's Lemma ensures inequality  \eqref{ine:time:deriv}.

To prove \eqref{ine:time:deriv:-ro}, we combine the inequality \eqref{ine:time:deriv} and the Poincaré inequality \eqref{Poincare:l:h}.
This finishes the proof of Proposition \ref{propo:TIme:derivative}.
\end{proof}
As a consequence we have the following results, the proof of which can be found in \cite{DanMUch13}[Lemma 3.4]. Since the proof does not depend on the $L^\infty$ bound on the density, it can be repeated directly as in \cite{DanMUch13}, hence we skip it.
\begin{lem}
    \label{lem:time:regular}
    Let $q\in [1,\infty)$ and $\alpha\in(0,1/2)$. Then $v\in H^{\frac{1}{2}-\alpha}(0,T;L^q)$ and the following estimate holds:
    \begin{align*}
        \lVert v\rVert_{H^{\frac{1}{2}-\alpha}(0,T:L^q)}^2\le  \lVert v \rVert_{L^2(0,T:L^q)}^2+ C_{\alpha,T}  \lVert\sqrt{t} v_t\rVert_{L^2(0,T:L^q)}^2,
    \end{align*}
    where $C_{\alpha,T}$ depending only on $\alpha$ and on $T$.
\end{lem}
\subsection{Shift of Integrability}
In the present section, our goal is to exploit the regularity of $ \n^2 v$ and $\n P$ that we have proved so far, in order to bound $\n v$ in $L^1(0,T;L^\infty)$. 
\begin{lem}
    \label{shif:regula}
   Let $(\ro, v)$ be a smooth enough solution to system (\textbf{INS}) on $[0,T^*)\times \T^2$  verifying estimates  \eqref{sob:ine:p} and \eqref{est_nabla2} with initial conditions $(\ro_0,v_0)$  satisfying
   \[  \ro_0\in\bigcap_{1\le p}L^p(\Omega),\;\; v_0\in H^1_0(\Omega),\;\; {\rm div\;}{v_0}=0 . \]
    For all $T\in  [0,T^*)$, $\eta\in [2,\infty]$ and $\varepsilon\in (0,1]$ , we have
   \begin{equation}\label{shift:regula:est} \|\nabla^2\sqrt{t}v\|_{L^\eta(0,T;L^{\lambda-\varepsilon})} + \|\nabla\sqrt{t}P\|_{L^\eta(0,T;L^{\lambda-\varepsilon})} \leq C_{0,T,\varepsilon} \quad \text{with} \quad \lambda = \frac{2\eta}{\eta-2}, \end{equation}
   where $C_{0,T}$ is a positive constant depending only on $\normede{\sqrt{\ro_0}v_0}$, $\normede{\n v_0}$ and $\|\ro_0\|_p$ for some $p$ depending on $\varepsilon$.

As  a direct consequence, for all $1\leq s<2$
it holds
\begin{align}
    \label{bound:nv:L1(Linfty)}
   \left( \int^T_0\normeinf{\n v(t)}^s dt\right)^{\frac{1}{s}}\le C_{0,T}.
\end{align}
\end{lem}

\begin{rmq}

    For $\ro_0\in\bigcap_{1\le p<\infty }L^p(\Omega)$, we obtain the same regularity as in \cite[Lemma 3.4]{DanMu} for $\ro_0\in L^\infty(\Omega)$. Assuming instead that $\ro_0\in L^p(\Omega)$ for some $p>3$ fixed, we have a different constraint, depending on $p$:
    \begin{align}
       \label{shift:regula:esti}
       \lVert \n^2\sqrt{t}v \rVert_{L^\eta(0,T;L^{\lambda-\varepsilon})}+ \lVert \n \sqrt{t} P\rVert_{L^\eta(0,T;L^{\lambda-\varepsilon})} \le C_{0,T}\esp{with} \lambda =\frac{2\eta p}{2+\eta p - 2p +\eta},
   \end{align}
   and respectively
   \begin{align}
    \label{bound:nv:L1(Linfty):p}
   \left( \int^T_0\normeinf{\n v(t)}^s dt\right)^{\frac{1}{s}}\le C_{0,T} \quad \text{for all} \quad 1\le s< \frac{2(p-1)^2}{(p-1)^2+p}. 
\end{align}
\end{rmq}

\begin{proof}[Proof of Lemma \ref{shif:regula}]
    The proof will based on the Proposition \ref{propo:TIme:derivative} and the following Stokes system:
    \begin{align}
        \label{Ss:system:}
        \begin{cases}
            &  -\Delta\sqrt{t} v+\n \sqrt{t}P=-\ro\sqrt{t} v_t-\sqrt{t}\ro v\cdot\n v \esp{in}(0,T)\times \T^2,\\
              & \dvg{\sqrt{t}v}=0  \esp{in}(0,T)\times \T^2.
        \end{cases}
    \end{align}

    From Proposition \ref{propo:TIme:derivative}, we know that $\sqrt{t\ro}v_t\in L^\infty(0,T;L^2)$. Since $\ro\in L^\infty(0,T;L^p)$ for any $p<\infty$, we deduce that $\ro\sqrt{t}v_t\in L^\infty(0,T;L^r)$ for all $r<2$. On the other hand, from the estimate on $\nabla\sqrt{t}v_t$ it follows that $\sqrt{t}v_t$ is bounded in $L^2(0,T;L^q)$ for any $q<\infty$ and thus $\ro\sqrt{t}v_t\in L^2(0,T;L^q)$ as well. In conclusion, by interpolation we get
    \begin{equation}
        \|\ro\sqrt{t}v_t\|_{L^\eta(0,T;L^{\lambda-\varepsilon})} \leq C_{0,T,\varepsilon} \quad \text{for any $\eta\geq 2$ and arbitrarily small} \quad \varepsilon>0,
    \end{equation}
    where $\lambda$ is as in (\ref{shift:regula:est}) and $C_{0,T,\varepsilon}$ depends on $\|\sqrt{\ro_0}v_0\|_2$, $\|\nabla v_0\|_2$ and the norm of $\ro_0$ in suitable Lebesgue space, depending on $\varepsilon$.



For the term $\sqrt{t}\ro v\cdot\nabla v$ we proceed similarly. By Remark \ref{rmq:conse:Poi:Sobo:reg}, $\sqrt{t}\ro v$ is bounded in $L^\infty(0,T;L^q)$ for any $q<\infty$.  Moreover, by the estimates on $\nabla v$ and $\nabla^2v$ from estimates \eqref{sob:ine:p} and \eqref{est_nabla2}, we get that $\nabla v$ is bounded in $L^\infty(0,T;L^2)$ and $L^2(0,T;L^q)$ for any $q<\infty$. In conclusion, $\sqrt{t}\ro v\cdot\nabla v$ is bounded in $L^\infty(0,T;L^r)$ and $L^\infty(0,T;L^q)$ for any $r<2$ and $q<\infty$. By interpolation, we have again
\[ \|\sqrt{t}\ro v\cdot\nabla v\|_{L^\eta(0,T;L^{\lambda-\varepsilon})} \leq C_{0,T,\varepsilon}. \]
Since we have shown that the right hand side of (\ref{Ss:system:}) belongs to $L^\eta(0,T;L^{\lambda-\varepsilon})$, we obtain the desired regularity for $\Delta\sqrt{t}v$ and $\nabla\sqrt{t}P$, which ends the first part of the Lemma.


For the second part of the Lemma, we proceed the same way as in \cite{DanMu}. Fix any $1\leq s<2$ and let $\eta$ be such that $s\eta<2(\eta-s)$. Since $\lambda>2$ for any $\eta\geq 2$, from the Sobolev embedding and Poincar\'e inequality we get
\[ \|\nabla\sqrt{t}v(t,\cdot)\|_{L^\infty} \leq C\|\nabla\sqrt{t}v(t,\cdot)\|_{W^{1,\lambda}} \leq C\|\nabla^2\sqrt{t}v(t,\cdot)\|_{L^\lambda}. \]

Therefore
\[\begin{aligned} \int_0^T\|\nabla v\|_{L^\infty}^s\;\dd t =& \int_0^Tt^{-\frac{s}{2}}\|\nabla\sqrt{t}v\|_{L^\infty}^s\;\dd t \leq C\int_0^T t^{-\frac{s}{2}}\|\nabla^2\sqrt{t}v\|_{L^{\lambda-\varepsilon}}^s\;\dd t \\
\leq & C\left(\int_0^T t^{-\frac{\eta s}{2(\eta-s)}}\;\dd t\right)^{1-\frac{s}{\eta}}\left(\int_0^T \|\nabla^2\sqrt{t}v\|_{L^{\lambda-\varepsilon}}^\eta\right)^{\frac{s}{\eta}}.
\end{aligned}\]
Since $\frac{\eta s}{2(\eta-s)}<1$, the right hand side is integrable and we get the bound for $\|\nabla v\|_{L^s(0,T;L^\infty)}$.

In the case for fixed $p$, we proceed analogously. First, note that we have $\lambda>2$ whenever $\eta<2(p-1)$. Choosing $1\leq s<\frac{2(p-1)}{p}$, let us pick $\frac{2s}{2-s}< \eta< 2(p-1)$. Then again $\frac{\eta s}{2(\eta-s)}<1$ and we can repeat the reasoning from the previous case.

\end{proof}

\section{Existence and uniqueness of solutions}

This part is dedicated to the mathematical justification of our analysis. The existence and the issue of uniqueness is clarified here. From the result \cite{DanMu} in Theorem 2.1, we know that for $\rho_0 \in L^\infty$ and $v_0 \in H^1$, global-in-time solutions exist. This is our basic tool to construct the approximative sequence. We fix an arbitrary time $T > 0$ in order to control the compactness of the sets under consideration.

Given $k \in \mathbb{N}$ we define
\begin{equation}
    \rho^k_0 = \min\{ \rho_0, k\}, \qquad v_0^k = v_0.
\end{equation}
Applying Theorem 2.1 from \cite{DanMu}, we obtain the global-in-time solutions $(\rho^k, v^k)$ initiated by $(\rho_0^k, v_0^k)$. The densities are globally bounded by $k$, and the velocity satisfies all the bounds from Theorem \ref{thm:Lp:density}, thanks to the obvious fact that
\begin{equation}
\|\rho_0^k\|_p \leq \|\rho_0\|_p \mbox{ \ \ for all } p \leq \infty, \qquad \|\sqrt{\rho^k_0} v_0\|_2 \leq 
\|\sqrt{\rho_0} v_0\|_2.
\end{equation}

What is left, is to justify the passage to the limit as $k \to \infty$. Since all the nonlinear terms in the equations contain the density multiplied by a function of velocity,
strong convergence is required only for the latter. However, from Lemma \ref{lem:time:regular} and the estimates from Theorem \ref{thm:Lp:density}, we obtain in particular that
\begin{equation}
    \|v^k\|_{H^\alpha(0,T;L^2)} + \|v^k\|_{L^2(0,T;H^1)} \leq C(\mbox{initial data})
\end{equation}
for any $\alpha<1/2$, with the right-hand side independent of $k$. Choosing suitable $\alpha$, this leads to the convergence 
\begin{equation}
    v^k \to v \mbox{ strongly in } L^{3}(0,T;L^3)
\end{equation}
Thus, having at hand
\begin{equation}
    \|\rho^k\|_{L^\infty(0,T;L^p)} \leq \|\rho_0\|_p,
\end{equation}
up to a subsequnece we have
\begin{equation}
    \rho^k \rightharpoonup \rho \mbox{ weakly-}\ast \mbox{ \ in \ } L^\infty(0,T;L^p).
\end{equation}
Taking $p>3$, we conclude that
\begin{equation}
    {\rm div} (\rho^k v^k \otimes v^k) \rightharpoonup 
    {\rm div} (\rho v \otimes v) \mbox{ \ \ in \ } \mathcal{D}'(\T^2 \times [0,T)).
\end{equation}
Analogously, $\rho^k v^k \rightharpoonup 
\rho v$. Consequently, the limit $(\rho,v)$ is the weak solution to the system (\ref{INS})
\begin{equation}
\begin{array}{l}
    \displaystyle \int_0^T \int_{\T^2} (\rho \psi_t + \rho u \cdot \nabla \psi )dxdt = \int_{\T^2} \rho_0(x)\psi(0,x)
    dx, \\[8pt]
   \displaystyle \int_0^T\int_{\T^2} (\rho u \phi_t +
   \rho u \otimes u :\nabla \phi - \mu \nabla u \nabla \phi )dxdt= \int_{\T^2} \rho_0(x) u_0(x) \phi(0,x) dx
\end{array}    
\end{equation}
for all $\psi \in \mathcal{D}(\T^2 \times [0,T))$ and 
all $\phi \in \mathcal{D}(\T^2 \times [0,T);\R^2)$ such that div $\phi =0$.

We maintain the regularity of $\nabla^2 u$ through weak convergence and the estimates from Theorem \ref{thm:Lp:density}. Thus, we obtain the solution as defined by Theorem \ref{thm:Lp:density}.

The final question is uniqueness. However, all the necessary elements to transform the system into Lagrangian coordinates are present. In particular, $\nabla v \in L^1(0,T;L^\infty)$. Therefore, we can apply directly Theorem 2.1 from \cite{DanMu}. The entirety of Section 4 applies to our case. With this, our Theorem \ref{thm:Lp:density} is complete.

Similar arguments lead to the proof of existence and uniqueness of solutions in the context of Theorem \ref{thm:Lp:ro>0}.

\subsubsection*{Acknowledgements}
The first author (JPA) has received funding from the European Union's Horizon 2020 research and innovation programme under the Marie Sk\l{}odowska-Curie grant agreement no. 945332.\\
The second author (PBM)  has been partly supported by the National Science Centre grant no. 2022/45/B/ST1/03432 (OPUS). The work of the third author (MS) was supported by the National Science Centre grant no. 2022/45/N/ST1/03900 (Preludium).

\begin{appendices}

\section{Appendix A}
For the reader's convenience, we recall here a few results that are used repeatedly in the paper. The first one is the following weighted Poincaré inequality 
\begin{lem}
    \label{poincare:inequa}
    Let $\ro:\T^2 \rightarrow \R$ be a nonnegative and nonzero measurable function. Then we have for all $ b\in H^1(\T^2)$ and $1\le m<\infty$, there exists a constant $C_m>0$ such that
   \begin{align}
    \label{Poincare:l:h}
    \normes{b}{m} \le \frac{1}{M}\left| \intT \ro b\right| + C_m \log^{\frac{1}{2}}\left( e+  \frac{\normede{\ro}}{M} \right)\normede{\n b}.
\end{align}
with $M$ the average of $\ro$:
$$ M:=\intT \ro.$$
Furthermore, for any  $ 1\le p\le \infty$, $0\le \alpha< p$, $1\le q<\frac{p}{\alpha}$ and $\beta \ge \frac{1}{q}-\frac{\alpha}{p}$, there exists an absolute constant $ C_{\beta,p,q}$ so that
\begin{align}
\label{poincare:al:be:p}
    \begin{split}
          \normes{ \ro^\alpha |b|^\beta}{q}
    \le C_{\beta,p,q} \normes{\ro}{p}^{\alpha }\frac{1}{M^\beta}\left| \intT \ro b\right|^\beta +  C_{\beta,p,q} \normes{\ro}{p}^{\alpha } \log^{\frac{\beta}{2}}\left( e+  \frac{\normede{\ro}}{M} \right)\normede{\n b}^\beta.
    \end{split}
\end{align}
\end{lem}

\begin{proof}
Before going into the details of the proof, let us notice that
inequality  \eqref{Poincare:l:h} has been used in the work  \cite{DanMu} on the two-dimensional incompressible Navier-Stokes equations with vacuum and with bounded density. For the reader convenience, we provide its proof here. 

We decompose $ b $ as follows:
\begin{align*}
    b=\overline{b}+ b_n+\widetilde{b}_n,
\end{align*}
where $\overline{b}$ is the average of $b$ and for any integer $n$ and $ x\in \T^2$,
\begin{align*}
  b_n(x):=\sum_{ \substack{k\in \mathbb{Z}^2\backslash  \{(0,0) \} \\   1\le |k|\le n} }\widehat{b}_k e^{2i\pi k\cdot x}   \esp{and}  \widetilde{b}_n(x) :=  \sum_{ \substack{k\in \mathbb{Z}^2\backslash \{(0,0)  \} \\ |k|> n}}\widehat{b}_k  e^{2i\pi k\cdot x}.
\end{align*}
The proof of \eqref{Poincare:l:h} relies on the following inequalities: for all $ 2\le q<\infty$,
\begin{align}
    \label{est:lf:b}
    \normeinf{b_n}\le C\sqrt{\log n} \normede{\n b} \esp{and} \normes{\widetilde{v}_n}{q}\le \frac{1}{n^{\frac{2}{q}}}\normede{\n b}.
\end{align}
Indeed, by Cauchy-Schwarz inequality and Parseval's identity, it is easy to prove that
\begin{align*}
     \normeinf{b_n}&\le  \sum_{ \substack{k\in \mathbb{Z}^2\backslash  \{(0,0) \} \\   1\le |k|\le n} }  \frac{1}{|k|}|k||\widehat{b}_k |\\
     &\le \left(\sum_{ \substack{k\in \mathbb{Z}^2\backslash  \{(0,0) \} \\   1\le |k|\le n} }  \frac{1}{|k|^2} \right)^{\frac{1}{2}} \left(\sum_{ \substack{k\in \mathbb{Z}^2\backslash  \{(0,0) \} \\   1\le |k|\le n} }  |k|^2|\widehat{b}_k |^2 \right)^{\frac{1}{2}}   \\
     &\le   C\sqrt{\log n} \normede{\n b}.
\end{align*}
On the other hand, for $2\le q<\infty$, since $ \Dot{H}^{1-\frac{2}{q}} \hookrightarrow   L^q(\T^2)$, we have
\begin{align*}
    \normes{\widetilde{b}_n}{q}\le C \lVert \widetilde{b}_n \rVert_{ \Dot{H}^{1-\frac{2}{q}} } \le C \frac{1}{n^{\frac{2}{q}}}\normede{\n b}.
\end{align*}
This complete the proof of \eqref{est:lf:b}.

Now, from Poincaré inequality, we have the obvious inequality: for all $1\le m< \infty$ 
\begin{align}
    \label{norm:m:b=f(n:v)}
    \normes{b}{m}\le |\overline{b}|+ C_m \normede{\n b}.
\end{align}
Because average of both $ b_n $ and  $ \widetilde{b}_n$ is $0$, one may write that, 
\begin{align*}
    M\overline{b}&=\intT\ro b+ \intT (M-\ro)(b-\overline{b})dx\\
    &=\intT\ro b-\intT \ro b_n-\intT\ro\widetilde{b}_n dx. 
\end{align*}
Therefore, using Hölder and Poincaré inequality, and also \eqref{est:lf:b},
\begin{align*}
    M|\overline{b}|&\le \left| \intT\ro b\right| +M\normeinf{b_n}+  \normede{\ro}\normede{\Tilde{b}_n}\\
   & \le   \left| \intT\ro b\right| +CM\left( \sqrt{\log n} + \frac{1}{n} \frac{\normede{\ro}}{M}\right)\normede{\n b}. 
\end{align*}
Then, taking $n \approx \frac{\normede{\ro}}{M}$, we have
\begin{align}
    \label{avera:control:b}
    |\overline{b}|\le \frac{1}{M}\left| \intT\ro b\right|+ C\log^\frac{1}{2}\left( e + \frac{\normede{\ro}}{M}\right)\normede{\n b}
\end{align}
and putting \eqref{avera:control:b} together with \eqref{norm:m:b=f(n:v)} gives \eqref{Poincare:l:h}.
  
Finally, to prove \eqref{poincare:al:be:p} we take advantage on Hölder's inequality.
Let $ 1\le p\le \infty$, $0<\alpha< p$, $1\le q<\frac{p}{\alpha}$ and $\beta \ge \frac{1}{q}-\frac{\alpha}{p}$. There holds

\begin{align*}
    \normes{ \ro^\alpha |b|^\beta}{q}\le  \normes{\ro^{\alpha }}{\frac{p}{\alpha }}\normes{|b|^\beta}{ p^*}\le \normes{\ro}{p}^\alpha\normes{b}{ \beta p^*}^\beta
\end{align*}
  with $\frac{1}{p^*}:= \frac{1}{q}-\frac{\alpha }{p}$ so that from the definition of $\beta$, we have $\beta p^*\ge 1$. 
Combining this inequality and \eqref{Poincare:l:h} gives \eqref{poincare:al:be:p}.
\end{proof}
We also used the following version of Desjardins' estimate  in order to propagate the $L^2$ norm of $\n v$. (see also \cite[Lemma 2]{DanMu})
\begin{lem}
\label{ladyn:log:cor:lem}
Let $p>1$.
    There exists a constant $C$ so that for all $b\in H^1(\T^2)$ and function $\ro\ge 0$ satisfying  $\ro \in L^p(\T^2)$, we have
    \begin{equation}
        \label{ladyn:ine:correc}
        \begin{split}
        \left( \intT\ro b^4 \right)^{\frac{1}{2}} &\le  2\normede{\sqrt{\ro}b}\left|\intT\ro b\right|\\
        &+\frac{p}{p-1} C\normede{\sqrt{\ro}b} \normede{\n b} \log^{\frac{1}{2}}\left( e+\frac{\normede{\ro}^2}{M^2}+\normep{\ro}\frac{\normede{\n b}^2}{\normede{\sqrt{\ro}b}^2}\right)
        \end{split}
    \end{equation}
    with $M:=\displaystyle \intT\ro$.
\end{lem}

\begin{proof}
 Fix some $n\in \mathbb{N}$. Then, keeping the same notation as in the above Lemma and using Hölder inequality, one has
  \begin{align*}
      \left( \intT \ro b^4 dx \right)^\frac{1}{2} &=\normede{ \sqrt{\ro}(\overline{b}+b_n+\widetilde{b}_n)b}\\
      &\le  \left(| \overline{b}|+\normeinf{ b_n }\right) \normede{ \sqrt{\ro}b} + \normep{\ro}^\frac{1}{4} \normes{\widetilde{b}_n}{q_4}\normes{\ro^\frac{1}{4} b}{4},
  \end{align*}
  where $\frac{1}{q_4}:= \frac{1}{4}-\frac{1}{4p} $.\\
  We  thus have, using Young inequality,
\begin{align*}
    \left( \intT \ro b^4 dx \right)^\frac{1}{2}\le 2   \left(| \overline{b}|+\normeinf{ b_n }\right) \normede{ \sqrt{\ro}b} + \normep{\ro}^\frac{1}{2} \normes{\widetilde{b}_n}{q_4}^2.
\end{align*}
Hence, taking advantage of estimates \eqref{est:lf:b}, we get
  \begin{align*}
      \left( \intT \ro b^4 dx \right)^\frac{1}{2} \le 2 | \overline{b}|\normede{ \sqrt{\ro}b} + C\normede{\sqrt{\ro}b}\left( \sqrt{\log n } + n^{-\frac{4}{q_4}} \normep{\ro}^\frac{1}{2} \frac{\normede{\n b}}{\normede{\sqrt{\ro}b}}\right) \normede{\n b}.
  \end{align*}
Taking for $n$ the closest positive integer to 
$$ \left( \normep{\ro}^\frac{1}{2} \frac{\normede{\n b}}{\normede{\sqrt{\ro}b}} \right) ^\frac{q_4}{4},$$ we end up with

\begin{align}
      \label{Desjar:ine:1}
      \left( \intT \ro b^4 dx \right)^\frac{1}{2} \le 2 | \overline{b}|\normede{ \sqrt{\ro}b} + C\frac{q_4}{4}\normede{\sqrt{\ro}b}\log^\frac{1}{2}   \left( e + \normep{\ro} \frac{\normede{\n b}^2}{\normede{\sqrt{\ro}b}^2} \right) 
      \normede{\n b}.
  \end{align}
Then, taking advantage on the control \eqref{avera:control:b} of $|b|$ and the definition of $q_4$ leads to \eqref{ladyn:ine:correc}. 

\end{proof}

\section{Properties of the space $\mathcal{L}$ and $\mathcal{Y}_0$.}\label{app_L}

Let us now discuss the classes $\mathcal{L}$ and $\mathcal{Y}_0$. In particular we are going to prove the embedding \eqref{emb:Y:L:BMO}. Let us first recall the definition of the space $L^{\exp}(\Omega)$:
\begin{align*}
   & L^{\exp}(\Omega) = \left\{f\colon\Omega\to\R: \exists_{\beta>0} \text{ such that } \int_\Omega e^{\frac{|f(x)|}{\beta}}-1 \;\dd x <\infty\right\}\cdotp
\end{align*}
Obviously we have $L^\infty(\Omega)\subset \mathcal{L}\subset \mathcal{Y}_0$. 
On the other hand, using the properties of the exponential function, we will show that $\mathcal{Y}_0\subset L^{\exp}(\Omega)$.

Fix any $\beta>0$. We have
\[\begin{aligned} \int_\Omega e^{\frac{|f(x)|}{\beta}}-1\;\dd x =& \int_\Omega\sum_{k=1}^\infty \frac{1}{k!}\frac{|f(x)|^k}{\beta^k}\;\dd x = \sum_{k=1}^\infty\frac{1}{k!}\frac{1}{\beta^k}\|f\|_k^k \\
\leq & \sum_{k=1}^\infty \frac{1}{k!}\left(\frac{k}{\beta} \frac{\|f\|_k}{k} \right)^k.
\end{aligned}\]
From Stirling's formula, for large $k$ we have
\[ \frac{1}{k!}\lesssim \frac{1}{\sqrt{2\pi k}}\left(\frac{e}{k}\right)^k. \]
Therefore 
\[ \frac{1}{k!} \biggl(\frac{k}{\beta}\frac{\|f\|_k}{k} \biggr)^k \lesssim \biggl(\frac{e }{\beta}\frac{\|f\|_k}{k} \biggr)^k . \]
Since $\lim_{k\to\infty}\frac{\|f\|_k}{k}=0$, for sufficiently large $k$ it holds $\frac{e }{\beta}\frac{\|f\|_k}{k}<\frac{1}{2}$. Therefore the series $\displaystyle\sum_{k=1}^\infty \frac{1}{k!}\frac{1}{\beta^k}\|f\|_k^k$ is convergent.

Now, let us show that the inclusions $L^\infty\subset\mathcal{L}\subset\mathcal{Y}_0$ are sharp. The examples below consider the situation where $\Omega$ is a $d$-dimensional ball, but they can be adjusted to any bounded domain (or the torus) by a suitable rescaling and extending the function by zero.

Below, we will use mulitple times the Stirling's formula for the Gamma function:
\begin{equation}\label{stirling} \Gamma(p+1) = \sqrt{2\pi p}\left(\frac{p}{e}\right)^p\left(1+O\left(\frac{1}{p}\right)\right)
\end{equation}

\begin{example}
The function $-\log|x|$ for $|x|\leq 1$ belongs to $L^{\exp}$ but not to $\mathcal{Y}_0$.
\end{example}
\begin{proof}
 The first statement is straightforward and comes from the fact that for any $\beta>\frac{1}{d}$ we have
 \[ \int_{B(0,1)} e^{\frac{-\log|x|}{\beta}}\;\dd x = \int_{B(0,1)}|x|^{\frac{-1}{\beta}}\;\dd x = c_d\int_0^1 r^{d-1-\frac{1}{\beta}}\;\dd r <\infty. \]
 On the other hand, by the explicit calculations
 \[\begin{aligned} \|\log|x|\mathbbm{1}_{B(0,1)} \|_p^p =& \int_{B(0,1)}(-\log|x|)^p\;\dd x = c_d\int_0^1 r^{d-1}(-\log r)^p\;\dd r = c_d\int_0^\infty e^{-ds}s^p\;\dd s \\
 =& \frac{c_d}{d^{p+1}}\int_0^\infty e^{-z}z^p\;\dd z = \frac{c_d}{d^{p+1}}\Gamma(p+1). \end{aligned}\]
 Using (\ref{stirling}), we know that
 \begin{equation}\label{log} \|\log|x|\mathbbm{1}_{B(0,1)}\|_p^p \sim \sqrt{p}\left(\frac{p}{de}\right)^p \end{equation}
 and thus $\|\log|x|\|_p\sim p$. Therefore $\log|x|\notin\mathcal{Y}_0$.
\end{proof}

 

\begin{example}
The unbounded function $\log\log|\log|x||$ for $|x|<e^{-e}$ belongs to $\mathcal{L}$.
\end{example}
\begin{proof}
 \textbf{Step 1.} First, we will show that
 \[ \|\log|\log|x||\mathbbm{1}_{B(0,e^{-e})}\|_p^p \lesssim \left(\frac{\log p}{e}\right)^p p\cdot\left(\frac{f(p)}{de}\right)^{f(p)}. \]
  Note that for any $a>0$,
 \[ a^p \leq \int_a^\infty s^pe^{a-s}\;\dd s \leq e^a\Gamma(p+1). \]
 Therefore, for $f(p)=\frac{p}{\log p}$, we have
 \[\begin{aligned} \int_{B(0,e^{-e})}(\log|\log|x||)^p\;\dd x =& \left(\frac{1}{f(p)}\right)^p\int_{B(0,e^{-e})}\big[f(p)\log|\log|x||)\big]^p\;\dd x \\
 =& \left(\frac{1}{f(p)}\right)^p\int_{B(0,e^{-e})}\left[\log\left(|\log|x||^{f(p)}\right)\right]^p\;\dd x \\
 \leq & \left(\frac{1}{f(p)}\right)^p\Gamma(p+1)\int_{B(0,e^{-e})} |\log|x||^{f(p)}\;\dd x. 
 \end{aligned}\]
 Using the formula (\ref{log}), if follows that 
 \[ \int_{B(0,e^{-e})} (\log|\log|x||)^p\;\dd x \lesssim \left(\frac{1}{f(p)}\right)^p\Gamma(p+1)\sqrt{f(p)}\left(\frac{f(p)}{de}\right)^{f(p)}. \]
 Using again the Stirling's formula (\ref{stirling}) and the definition of $f$ (in particular $f(p)<p$), we get
 \[ \biggl\|\log|\log|x||\mathbbm{1}_{B(0,e^{-e})}\biggr\|_p^p \lesssim \left(\frac{\log p}{e}\right)^p p\cdot\left(\frac{f(p)}{de}\right)^{f(p)}. \]
 
 \textbf{Step 2.} Now, we use the above information to estimate the $L^p$ norm of $\log\log|\log|x||$. Let $g(p)=\frac{p}{(\log p)^{1/2-\varepsilon}}$ for some $0<\varepsilon<1/2$. Similarly as before, we have
 \[\begin{aligned} \int_{B(0,e^{-e})}(\log\log|\log|x||)^p\;\dd x =& \left(\frac{1}{g(p)}\right)^p\int_{B(0,e^{-e})}\left[\log\left((\log|\log|x||)^{g(p)}\right)\right]^p\;\dd x \\
 \leq & \left(\frac{1}{g(p)}\right)^p\Gamma(p+1)\int_{B(0,e^{-e})}(\log|\log|x||)^{g(p)}\;\dd x.
 \end{aligned}\]
 From Step 1, we know that
 \[ \int_{B(0,e^{-e})}(\log|\log|x||)^{g(p)}\;\dd x \lesssim \left(\frac{\log g(p)}{e}\right)^{g(p)} g(p)\cdot\left(\frac{f(g(p))}{de}\right)^{f(g(p))}. \]
 Therefore, using again (\ref{stirling}) and the fact that $g(p)<p$, we get
 \[ \|\log\log|\log|x||\|_p \lesssim \frac{p}{g(p)}\left(\frac{\log p}{e}\right)^{\frac{g(p)}{p}} (p\sqrt{p})^{1/p}\cdot\left(\frac{f(p)}{de}\right)^{\frac{f(p)}{p}}. \]
 In conclusion, since $\frac{p}{g(p)}=(\log p)^{1/2-\varepsilon}$ and
 \[ \lim_{p\to\infty}\left[\left(\frac{\log p}{e}\right)^{\frac{g(p)}{p}} (p\sqrt{p})^{1/p}\right] = 0, \quad \lim_{p\to\infty}\left(\frac{f(p)}{de}\right)^{\frac{f(p)}{p}} = e, \]
 we get
 \[ \frac{1}{\sqrt{\log p}}\|\log\log|\log|x||\|_p\log^{1/2}(e+\|\log\log|\log|x||\|_p) \to 0 \]
 as $p\to\infty$. This proves that $\log\log|\log p||\in \mathcal{L}$.
\end{proof}

\end{appendices}

\bibliography{bib}
\bibliographystyle{plain}

\end{document}